\documentclass[reqno,11pt]{amsart}

\usepackage{a4wide}
\usepackage{color}
\usepackage{mathrsfs}
\usepackage{mathtools}
\usepackage{amsmath}
\usepackage{amssymb}
\usepackage{esint}
\numberwithin{equation}{section}
\usepackage[colorlinks,citecolor=green,linkcolor=red]{hyperref}

\usepackage[latin1]{inputenc}

\newcommand{\nchi}{{\raise.3ex\hbox{\(\chi\)}}}
\newcommand{\N}{\mathbb{N}}
\newcommand{\Z}{\mathbb{Z}}
\newcommand{\R}{\mathbb{R}}
\newcommand{\sfd}{{\sf d}}
\renewcommand{\d}{{\mathrm d}}
\newcommand{\X}{{\rm X}}
\newcommand{\Y}{{\rm Y}}
\newcommand{\mm}{\mathfrak{m}}
\newcommand{\lip}{{\rm lip}}

\newcommand{\Tan}{{\rm Tan}}
\newcommand{\spt}{{\rm spt}}

\newcommand{\fr}{\penalty-20\null\hfill\(\blacksquare\)}

\newtheorem{theorem}{Theorem}[section]

\newtheorem{lemma}[theorem]{Lemma}
\newtheorem{proposition}[theorem]{Proposition}
\newtheorem{definition}[theorem]{Definition}
\newtheorem{remark}[theorem]{Remark}
\newtheorem{question}[theorem]{Question}

\title{Non-Hilbertian tangents to Hilbertian spaces}

\author[Danka Lu\v{c}i\'{c}]{Danka Lu\v{c}i\'{c}}
\address[Danka Lu\v{c}i\'{c}]{Universit\`{a} di Pisa,
Dipartimento di Matematica, Largo Bruno Pontecorvo 5,
56127 Pisa, Italy}
\email{danka.lucic@dm.unipi.it}

\author[Enrico Pasqualetto]{Enrico Pasqualetto}
\address[Enrico Pasqualetto]{Scuola Normale Superiore, Piazza dei
Cavalieri, 7, 56126 Pisa, Italy.}
\email{enrico.pasqualetto@sns.it}

\author[Tapio Rajala]{Tapio Rajala}
\address[Tapio Rajala]{University of Jyvaskyla, Department of Mathematics and Statistics, P.O.\ Box 35 (MaD), FI-40014
University of Jyvaskyla, Finland}
\email{tapio.m.rajala@jyu.fi}
\begin{document}
\date{\today}
\begin{abstract}
We provide examples of infinitesimally Hilbertian, rectifiable, Ahlfors regular metric
measure spaces having pmGH-tangents that are not infinitesimally Hilbertian.
\end{abstract}
\maketitle
\section{Introduction}
In the theory of metric measure spaces, one of the central themes is the investigation of the infinitesimal structure of the space under consideration,
which can be examined from different perspectives. On the one hand, an analytic approach consists in studying the behaviour of weakly differentiable
functions, which make perfectly sense even in this non-smooth framework thanks to \cite{Cheeger00,Shanmugalingam00,AmbrosioGigliSavare11},
where (equivalent) notions of a first-order Sobolev space have been introduced. On the other hand, a geometric viewpoint suggests that one looks at
the tangent spaces, obtained by taking limits of the rescalings of the space with respect to a suitable notion of convergence, typically induced by
the pointed measured Gromov--Hausdorff (pmGH) topology \cite{Edwards75,Gromov81,Fukaya87} or some of its variants. However, in full generality these
objects (scilicet, Sobolev functions and pmGH-tangents) may have little to do with the properties of the underlying space, as simple examples show.
Fortunately, the situation greatly improves under appropriate regularity assumptions. An instance of this fact is given by the class of PI spaces,
which are doubling metric measure spaces supporting a weak Poincar\'{e} inequality in the sense of Heinonen--Koskela \cite{HeinonenKoskela98}.
Indeed, as an outcome of Cheeger's results in \cite{Cheeger00}, we know that it is possible to develop a satisfactory first-order differential calculus
in PI spaces, as the latter verify a generalised form of Rademacher Theorem (concerning the almost everywhere differentiability of Lipschitz functions).
Additionally, every point of a PI space has non-empty pmGH-tangent cone (thanks to Gromov Compactness Theorem) and each pmGH-tangent is a PI
space itself (see \cite[Theorem 11.6.9]{HKST15}). In the literature, also the larger class of the so-called Lipschitz differentiability
spaces (LDS), \emph{i.e.}\ where the conclusion of Cheeger's Differentiation Theorem is satisfied, has been thoroughly
investigated, see \emph{e.g.}\ \cite{Bate15} after \cite{Kei:04,Keith04}. It was proved in \cite{Schioppa16} that pmGH-tangents to LDS are LDS,
but such tangents might be quite wild (for instance, they can be disconnected a.e., see \cite{Schioppa16-2}).
Let us also remark that under slightly stronger assumptions, namely for RNP differentiability spaces (where the LDS condition is required for Lipschitz
functions with values in Banach spaces satisfying the Radon--Nikod\'{y}m property) the tangents behave much better than for LDS spaces \cite{Eriksson-Bique2019}.
\smallskip

The present paper focuses on the properties of the pmGH-tangents to those metric measure spaces which `are Hilbertian at infinitesimal scales'.
In this regard, the relevant notion is called infinitesimal Hilbertianity \cite{Gigli12}. This assumption simply states that the 2-Sobolev space
is a Hilbert space and is very natural when dealing with various non-smooth generalisations of Riemannian manifolds, such as Alexandrov spaces or
\(\sf RCD\) spaces. Since the infinitesimal Hilbertianity condition concerns the differentials of Sobolev functions, one can expect it to be stable
only in some specific circumstances. As an indicator of this issue, just consider the fact that every metric measure space can be realised as the
pmGH-limit of a sequence of discrete spaces. A significant example of the stability of infinitesimal Hilbertianity is that of \(\sf RCD\) spaces,
which are infinitesimally Hilbertian spaces verifying the \({\sf CD}(K,N)\) condition, which imposes a lower bound \({\rm Ric}\geq K\) on the Ricci
curvature and an upper bound \({\rm dim}\leq N\) on the dimension, in some synthetic form. It follows from \cite{AmbrosioGigliSavare11-2} that
the class of \({\sf RCD}(K,N)\) spaces is closed under pmGH-convergence. Heuristically, even though the pmGH-convergence is a zeroth-order concept while the
Hilbertianity is a first-order one, the stability of the latter is enforced by the uniformity at the level of the second-order structure (encoded
in the common lower Ricci bounds). Another example can be found in \cite{LucicPasqualetto21}, where it is shown that the infinitesimal Hilbertianity is
preserved along sequences of metric measure spaces where the measure is fixed, while distances monotonically converge from below to the limit one.
In this case, the stability is in force for arbitrary metric measure spaces (with no additional regularity, such as \(\sf RCD\) spaces), but the
notion of convergence is much stronger than the pmGH one.
\smallskip

The problem we address in this paper is the following: \emph{given an infinitesimally Hilbertian metric measure space
(which fulfills further regularity assumptions), are its pmGH-tangents infinitesimally Hilbertian as well?}
The case of \({\sf RCD}(K,N)\) spaces is already settled, as a consequence of the previous discussion. Indeed, if \((\X,\sfd,\mm)\) is an
\({\sf RCD}(K,N)\) space, then for any radius \(r>0\) the rescaled space \((\X,\sfd/r,\mm_x^r,x)\), where \(\mm_x^r\) is the normalised measure
\[
\mm_x^r\coloneqq\frac{\mm}{\mm(B_r(x))},
\]
is an \({\sf RCD}(r^2 K,N)\) space.
In particular, each pmGH-tangent to \((\X,\sfd,\mm)\) at \(x\) (whose existence is guaranteed by Gromov Compactness Theorem)
is an \({\sf RCD}(0,N)\) space, thanks to the stability of the \(\sf RCD\) condition. Nevertheless, besides \(\sf RCD\) spaces, we will mostly obtain negative results.
Before passing to the statement of our first result, we need to fix some more terminology. Given a point \(x\in{\rm spt}(\mm)\) of a metric measure space \((\X,\sfd,\mm)\),
we denote by \({\rm Tan}_x(\X,\sfd,\mm)\) its pmGH-tangent cone. Moreover, we say that a metric measure space \((\X,\sfd,\mm)\) is \(\mm\)-rectifiable
provided it can be covered \(\mm\)-a.e.\ by Borel sets \(\{U_i\}_{i\in\N}\) that are biLipschitz equivalent to Borel sets in \(\R^{n_i}\) and satisfy
\(\mm|_{U_i}\ll\mathcal H^{n_i}\); notice that we are not requiring that \(\{n_i\}_{i\in\N}\subset\N\) is a bounded sequence. Under a (pointwise)
doubling assumption, the \(\mm\)-rectifiability requirement entails a very rigid behaviour of the pmGH-tangents, which are almost everywhere unique
and consist of a finite-dimensional Banach space, whose norm can be computed by looking at the blow-ups of the chart maps, together with the induced
(normalised) Hausdorff measure; see Proposition \ref{prop:tan_to_rect}. Nevertheless, the ensuing result holds:
\begin{theorem}\label{thm:example_non_PI}
There exists an infinitesimally Hilbertian, \(\mm\)-rectifiable, Ahlfors regular metric measure space \((\X,\sfd,\mm)\) such that
for \(\mm\)-a.e.\ point \(x\in\X\) the tangent cone \(\Tan_x(\X,\sfd,\mm)\) contains a unique, infinitesimally non-Hilbertian element.
\end{theorem}
We will prove Theorem \ref{thm:example_non_PI} in Section \ref{s:proof_thm1}. The key idea behind its proof is to construct a space whose `analytic dimension'
is zero (or one), so that the associated Sobolev space is necessarily Hilbert, but whose pmGH-tangents are two-dimensional and not Hilbertian. This kind of
situation is possible because we are not requiring the validity of a weak Poincar\'{e} inequality. In fact, when dealing with \(\mm\)-rectifiable PI spaces,
the situation improves considerably, as we will see later on. However, even in this case there can exist infinitesimally non-Hilbertian pmGH-tangents to
infinitesimally Hilbertian spaces:
\begin{theorem}\label{thm:example_PI}
There exists an infinitesimally Hilbertian, \(\mm\)-rectifiable, Ahlfors regular PI space \((\X,\sfd,\mm)\) such that
\(\Tan_{\bar x}(\X,\sfd,\mm)\) contains an infinitesimally non-Hilbertian element for some point \(\bar x\in\X\).
In addition, one can require that \((\X\setminus\{\bar x\},\sfd,\mm)\) is a Riemannian manifold.
\end{theorem}
The proof of Theorem \ref{thm:example_PI} is more involved and will be carried out in Section \ref{s:proof_thm2}. Roughly speaking, the strategy is to define
a Riemannian metric on \(\R^2\setminus\{0\}\) of the form \(\rho|\cdot|\), where \(\rho\) is a smooth function which is discontinuous at \(0\), so that its
induced length distance behaves like the \(\ell^1\)-norm when we zoom the space around \(0\). This way, we obtain an infinitesimally Hilbertian space whose
pmGH-tangent at \(0\) is not.
\begin{remark}{\rm
One might wonder in Theorem \ref{thm:example_PI} what can be said about the curvature of the Riemannian metric outside \(\{\bar x\}\). Suppose, for instance,
that \((\X,\sfd)\) is Ahlfors \(n\)-regular and has Ricci-curvature lower bound \(\kappa(x)\) in a neighbourhood of \(x\in\X\setminus\{\bar x\}\). Then, if
\(\kappa\in L^p(\X)\) with \(p>n/2\), by the scaling of the Ricci-curvature lower bound, we have that 
\[
\int_{(\X,\lambda\sfd)} |\kappa_\lambda(x)|^p\,\d\mathcal H^n(x) = 
\lambda^{n-2p}\int_{(\X,\sfd)} |\kappa(x)|^p\,\d\mathcal H^n(x) \to 0
\]
as \(\lambda \to \infty\), where \(\kappa_\lambda(x)\) is the Ricci-curvature lower bound at \(x\) for the blow-up \((\X,\lambda\sfd,\bar x)\). Consequently,
any measured Gromov--Hausdorff tangent at \(\bar x\) will be a flat space \cite{PetersenWei97} and in particular, the tangent will be infinitesimally Hilbertian.

It is not difficult to see that by scaling down the construction steps depending on their curvature lower bounds, the construction for Theorem \ref{thm:example_PI}
in an \(n\)-dimensional space, \(n \geq 3\), could be modified to have an integrable Ricci-curvature lower bound \(\kappa \in L^p(\X)\) for any given \(p < n/2\).
Notice that in general any compact length space can be approximated in the Gromov--Hausdorff distance by compact manifolds having \(L^{n/2}\)-integrable curvature,
see \cite{Aubry07}.

We also remark that the stability of integrable Ricci-curvature lower bounds in metric measure spaces, for \({\sf CD}(\kappa,N)\) spaces are known for continuous
lower bounds \(\kappa \in L^p\) with \(p > N/2\), see \cite{Ketterer21}.
\fr}\end{remark}
However, the phenomenon observed in Theorem \ref{thm:example_PI} cannot take place in a set of points of positive measure. This is the content of the next result:
\begin{theorem}\label{thm:generic_tan_PI_rect}
Let \((\X,\sfd,\mm)\) be an infinitesimally Hilbertian, \(\mm\)-rectifiable PI space. Then for \(\mm\)-a.e.\ point \(x\in\X\)
the unique element of \(\Tan_x(\X,\sfd,\mm)\) is infinitesimally Hilbertian.
\end{theorem}
Section \ref{s:proof_thm3} will be devoted to the proof of Theorem \ref{thm:generic_tan_PI_rect}, which is in fact only a combination of several
results that were already available in the literature.
\medskip

One might wonder whether the \(\mm\)-rectifiability assumption in Theorem \ref{thm:generic_tan_PI_rect} can be dropped.
In other words, a natural question is the following: 
\begin{question}\label{question:generic_PI}
{\rm Is it true that if \((\X,\sfd,\mm)\) is an infinitesimally
Hilbertian PI space, then \(\Tan_x(\X,\sfd,\mm)\) contains only infinitesimally Hilbertian elements for \(\mm\)-a.e.\ \(x\in\X\)?}
\end{question}
We are currently unable to address this question, thus we leave it as an open problem. We point out that a key result on PI spaces,
concerning the relation between the (analytic) differential structure and the (geometric) pmGH-tangents, is \cite[Theorem 1.12]{CheegerKleinerSchioppa16}
(see also the related earlier results in \cite{Cheeger00}). This result says, roughly, that for \(\mm\)-a.e.\ point \(x\) of a PI space \((\X,\sfd,\mm)\)
and for every pmGH-tangent \((\Y,\sfd_\Y,\mm_\Y,q)\in\Tan_x(\X,\sfd,\mm)\), one can construct a pointed blow-up map \(\hat\varphi\colon\Y\to T_x\X\) (obtained by
rescaling a chart \(\varphi\)) which is a metric submersion. Here, \(T_x\X\) stands for a fiber of the tangent bundle \(T\X\),
in the sense of Cheeger \cite{Cheeger00}. In the case where \((\X,\sfd,\mm)\) is infinitesimally Hilbertian, we have that \(T_x\X\)
is Hilbert for \(\mm\)-a.e.\ \(x\in\X\), see \cite[Section 2.5\textbf{(3)}]{Gigli14}. Nevertheless, since \(\hat\varphi\colon\Y\to T_x\X\)
is only a (possibly non-injective) metric submersion, there might be more independent directions on the tangent $\Y$ and thus
we cannot deduce from this information that the space \((\Y,\sfd_\Y,\mm_\Y)\) is infinitesimally Hilbertian.
It is an open problem whether the tangent spaces can really have more independent directions on positively many points.
A negative answer to this would resolve Question \ref{question:generic_PI} in the affirmative.
\smallskip

We conclude by mentioning that similar results hold also if one considers pmGH-asymptotic cones instead of pmGH-tangent cones. We commit the discussion on the relation
between infinitesimal Hilbertianity and asymptotic cones to Appendix \ref{app:asympt_cones}.
\subsection*{Acknowledgements}
The authors thank Adriano Pisante for having suggested the problem, as well as Nicola Gigli,
Sylvester Eriksson-Bique, and Elefterios Soultanis for the useful comments and discussions.
The first named author acknowledges the support from the project 2017TEXA3H ``Gradient flows, Optimal Transport and Metric Measure Structures'', funded by the Italian
Ministry of Research and University. The second named author acknowledges the support from the Balzan project led by Luigi Ambrosio.
The third named author acknowledges the support from the Academy of Finland, grant no.\ 314789.
\section{Preliminaries}
Let us begin by fixing some general terminology. For any exponent \(p\in[1,\infty]\), we denote by \(\|\cdot\|_p\) the \(\ell^p\)-norm on \(\R^n\),
namely, for every vector \(v=(v_1,\ldots,v_n)\in\R^n\) we define
\[
\|v\|_p\coloneqq\left\{\begin{array}{ll}
\big(|v_1|^p+\ldots+|v_n|^p\big)^{1/p},\\
\max\big\{|v_1|,\ldots,|v_n|\big\},
\end{array}\quad\begin{array}{ll}
\text{ if }p<\infty,\\
\text{ if }p=\infty.
\end{array}\right.
\]
For brevity, we will often write \(|\cdot|\) in place of \(\|\cdot\|_2\). The Euclidean distance on \(\R^n\) will be denoted
by \(\sfd_{\rm Eucl}(v,w)\coloneqq|v-w|\). By \(\mathcal L^n\) we mean the Lebesgue measure on \(\R^n\). Given an arbitrary
metric space \((\X,\sfd)\), we indicate with \(B_r(x)\), or with \(B_r^\sfd(x)\), the open ball in \((\X,\sfd)\) of radius
\(r>0\) and center \(x\in\X\). For any \(k\in[0,\infty)\), the \(k\)-dimensional Hausdorff measure on \((\X,\sfd)\) will be
denoted by \(\mathcal H^k\), or by \(\mathcal H^k_\sfd\).
\subsection{Metric measure spaces}
By a \textbf{metric measure space} \((\X,\sfd,\mm)\) we mean a complete, separable metric space \((\X,\sfd)\), together
with a boundedly-finite, Borel measure \(\mm\geq 0\) on \(\X\). Several (equivalent) notions of \textbf{Sobolev space}
over \((\X,\sfd,\mm)\) have been investigated in the literature, see for instance \cite{Cheeger00,Shanmugalingam00,AmbrosioGigliSavare11}
and \cite{AmbrosioGigliSavare11-3} for the equivalence between them. We follow the approach by Cheeger \cite{Cheeger00}: we declare that
a given function \(f\in L^2(\X)\) belongs to the Sobolev space \(H^{1,2}(\X)\) provided there exists a sequence \((f_n)_{n\in\N}\) of
boundedly-supported Lipschitz functions \(f_n\colon\X\to\R\) such that \(f_n\to f\) in \(L^2(\X)\) and \(\sup_{n\in\N}\int\lip(f_n)^2\,\d\mm<+\infty\),
where the \textbf{slope} function \(\lip(f_n)\colon\X\to[0,+\infty)\) is defined as
\[
\lip(f_n)(x)\coloneqq\limsup_{y\to x}\frac{\big|f_n(y)-f_n(x)\big|}{\sfd(y,x)},\quad\text{ if }x\in\X\text{ is an accumulation point,}
\]
and \(\lip(f_n)(x)\coloneqq 0\) otherwise. The \textbf{Sobolev norm} of a function \(f\in H^{1,2}(\X)\) is defined as
\[
\|f\|_{H^{1,2}(\X)}\coloneqq\bigg(\int|f|^2\,\d\mm+\inf_{(f_n)_n}\liminf_{n\to\infty}\int\lip(f_n)^2\,\d\mm\bigg)^{1/2},
\]
where the infimum is taken among all sequences \((f_n)_{n\in\N}\) of boundedly-supported Lipschitz functions converging to \(f\) in \(L^2(\X)\).
The resulting space \(\big(H^{1,2}(\X),\|\cdot\|_{H^{1,2}(\X)}\big)\) is a Banach space. The term \textbf{infinitesimally Hilbertian}, coined by
Gigli in \cite{Gigli12}, is reserved to those metric measure spaces whose Sobolev space is Hilbert.
By analogy, we say that a metric measure space is \textbf{infinitesimally non-Hilbertian} when its Sobolev space is not Hilbert.
\smallskip

We say that \((\X,\sfd,\mm)\) is \textbf{\(C\)-doubling}, for some \(C\geq 1\), provided \(\mm\big(B_{2r}(x)\big)\leq C\mm\big(B_r(x)\big)\)
holds for every \(x\in\X\) and \(r>0\). We say that \((\X,\sfd,\mm)\) is \textbf{\(k\)-Ahlfors regular}, for some \(k\geq 1\), if there
exists \(\alpha\geq 1\) such that \(\alpha^{-1}r^k\leq\mm\big(B_r(x)\big)\leq\alpha r^k\) for every \(x\in\X\) and \(r\in\big(0,{\rm diam}(\X)\big)\),
where \({\rm diam}(\X)\) stands for the diameter of \(\X\). Observe that each Ahlfors regular space is in particular doubling.
Moreover, we say that \((\X,\sfd,\mm)\) supports a \textbf{weak \((1,2)\)-Poincar\'{e} inequality} provided there exist \(C>0\) and \(\lambda\geq 1\)
such that for any boundedly-supported Lipschitz function \(f\colon\X\to\R\) it holds that
\[
\fint_{B_r(x)}\Big|f-{\textstyle\fint_{B_r(x)}f\,\d\mm}\Big|\,\d\mm\leq C r\bigg(\fint_{B_{\lambda r}(x)}\lip(f)^2\,\d\mm\bigg)^{1/2},
\quad\text{ for every }x\in\X\text{ and }r>0.
\]
Finally, by a \textbf{PI space} we mean a doubling space supporting a weak \((1,2)\)-Poincar\'{e} inequality.
For a thorough account of PI spaces, we refer to \cite{HKST15,Bjorn-Bjorn11} and the references therein.
\subsection{Length distances induced by a metric}
Let \(\rho\colon C\to(0,+\infty)\) be a given function, where \(C\subset\R^n\) is any convex set.
Then we denote by \(\sfd_\rho\), or by \(\sfd_\rho^C\), the length distance on \(C\) induced
by the metric \(C\times\R^n\ni(x,v)\mapsto\rho(x)|v|\). Namely, we define
\[
\sfd_\rho(a,b)\coloneqq\inf_\gamma\ell_\rho(\gamma),\quad\text{ for every }a,b\in C,
\text{ where }\ell_\rho(\gamma)\coloneqq\int_0^1\rho(\gamma_t)|\dot\gamma_t|\,\d t,
\]
while the infimum is taken among all Lipschitz curves \(\gamma\colon[0,1]\to C\) with \(\gamma_0=a\) and \(\gamma_1=b\).
Observe that if \(\alpha\leq\rho\leq\beta\) for some \(\beta>\alpha>0\), then
\(\alpha\,\sfd_{\rm Eucl}\leq\sfd_\rho\leq\beta\,\sfd_{\rm Eucl}\) on \(C\times C\).
\begin{lemma}\label{lem:d_rho_iH}
Let \(\rho\colon\R^n\to[\alpha,\beta]\) be given. Suppose \(\rho\) is continuous at \(x\in\R^n\). Then
\begin{equation}\label{eq:d_rho_iH}
\lim_{r\searrow 0}\frac{\sfd_\rho(x+rv,x)}{r}=\rho(x)|v|,\quad\text{ for every }v\in\R^n.
\end{equation}
In particular, if \(\mm\geq 0\) is a Radon measure on \(\R^n\) with \(\mm\ll\mathcal L^n\) and \(\rho\) is \(\mm\)-a.e.\ continuous,
then the metric measure space \((\R^n,\sfd_\rho,\mm)\) is infinitesimally Hilbertian.
\end{lemma}
\begin{proof}
The identity in \eqref{eq:d_rho_iH} is trivially verified at \(v=0\), so let us assume that \(v\neq 0\). Then:\\
{\color{blue}\(\leq:\)} Fix any \(\varepsilon>0\). Choose some \(\bar r>0\) satisfying \(\rho(x+rv)\leq\rho(x)+\varepsilon\) for every \(r\in(0,\bar r)\).
Calling \(\gamma^r\colon[0,1]\to\R^n\) the constant-speed parametrisation of the interval \([x,x+rv]\), one has
\[
\limsup_{r\searrow 0}\frac{\sfd_\rho(x+rv,x)}{r}\leq\limsup_{r\searrow 0}\frac{\ell_\rho(\gamma^r)}{r}
=\limsup_{r\searrow 0}\int_0^1\rho(x+rsv)|v|\,\d s\leq\big(\rho(x)+\varepsilon\big)|v|,
\]
whence it follows (by letting \(\varepsilon\searrow 0\)) that \(\limsup_{r\searrow 0}\sfd_\rho(x+rv,x)/r\leq\rho(x)|v|\).\\
{\color{blue}\(\geq:\)} Fix any \(\delta>\beta/\alpha\). For any \(r>0\) we have that \(\sfd_\rho(x+rv,x)\leq\beta r|v|\), while any Lipschitz curve
\(\gamma\colon[0,1]\to\R^2\) with \(\gamma_0=x\) that intersects \(\R^2\setminus B_{\delta r|v|}^{|\cdot|}(x)\) satisfies \(\ell_\rho(\gamma)\geq\alpha\delta r|v|\). Then
\begin{equation}\label{eq:d_rho_iH_aux}
\sfd_\rho(x+rv,x)=\inf\Big\{\ell_\rho(\gamma)\;\Big|\;\gamma\colon[0,1]\to B_{\delta r|v|}^{|\cdot|}(x)\text{ Lipschitz},\,\gamma_0=x,\,\gamma_1=x+rv\Big\}.
\end{equation}
Now fix any \(\varepsilon>0\). Choose some \(\bar r>0\) satisfying \(\rho(y)\geq\rho(x)-\varepsilon\) for every \(y\in B_{\delta\bar r|v|}^{|\cdot|}(x)\).
Hence, given \(r\in(0,\bar r)\) and \(\gamma\colon[0,1]\to B_{\delta r|v|}^{|\cdot|}(x)\) Lipschitz with \((\gamma_0,\gamma_1)=(x,x+rv)\), one has
\[
\ell_\rho(\gamma)=\int_0^1\rho(\gamma_t)|\dot\gamma_t|\,\d t\geq\big(\rho(x)-\varepsilon\big)\int_0^1|\dot\gamma_t|\,\d t=\big(\rho(x)-\varepsilon\big)r|v|.
\]
By recalling \eqref{eq:d_rho_iH_aux}, we can conclude that \(\liminf_{r\searrow 0}\sfd_\rho(x+rv,x)/r\geq\big(\rho(x)-\varepsilon\big)|v|\) and thus accordingly
that \(\liminf_{r\searrow 0}\sfd_\rho(x+rv,x)/r\geq\rho(x)|v|\), thanks to the arbitrariness of \(\varepsilon>0\).

All in all, the identity in \eqref{eq:d_rho_iH} is proved. Finally, let us pass to the verification of the last part of the statement. Suppose that \(\mm\) is a Radon
measure on \(\R^n\) with \(\mm\ll\mathcal L^n\) and that \(\rho\) is continuous at \(\mm\)-a.e.\ point of \(\R^n\). In particular, the space \((\R^n,\sfd_\rho,\mm)\)
is \(\mm\)-rectifiable and admits \(\big\{(\R^n,{\rm id}_{\R^n})\big\}\) as an atlas. Consequently, \eqref{eq:d_rho_iH} gives \(\|\cdot\|_x=\rho(x)|\cdot|\)
for \(\mm\)-a.e.\ \(x\in\R^n\), so that \((\R^n,\sfd_\rho,\mm)\) is infinitesimally Hilbertian (cf.\ \eqref{eq:def_ptwse_norm_rect} and Proposition
\ref{prop:equiv_iH_rect}), as desired.
\end{proof}
It is worth pointing out that the absolute continuity assumption in the last part of the statement of Lemma \ref{lem:d_rho_iH} might be dropped.
However, the present formulation of Lemma \ref{lem:d_rho_iH} is easier to achieve and sufficient for our purposes.
\subsection{Tangent cones}
In this paper we are concerned with \textbf{tangent cones}, considered with respect to the \textbf{pointed measured Gromov--Hausdorff} topology,
for whose definition we refer to \cite[Definition 3.24]{GigliMondinoSavare13}. By a \textbf{pointed metric measure space} \((\X,\sfd,\mm,x)\)
we mean a metric measure space \((\X,\sfd,\mm)\), together with a reference point \(x\in{\rm spt}(\mm)\), where \({\rm spt}(\mm)\subset\X\)
stands for the support of the measure \(\mm\). Given any radius \(r>0\), we denote by
\[
\mm^r_x\coloneqq\frac{\mm}{\mm\big(B_r(x)\big)}
\]
the \textbf{normalised measure} of scale \(r\) around \(x\).
\begin{definition}[Tangent cone]\label{def:tan_cone}
Let \((\X,\sfd,\mm,p)\) be a pointed metric measure space. Then we say that a given pointed metric measure space \((\Y,\sfd_\Y,\mm_\Y,q)\)
belongs to the \textbf{pmGH-tangent cone} \(\Tan_p(\X,\sfd,\mm)\) to \((\X,\sfd,\mm)\) at \(p\) provided there exists a sequence of radii \(r_k\searrow 0\) such that
\[
(\X,\sfd/r_k,\mm^{r_k}_p,p)\to(\Y,\sfd_\Y,\mm_\Y,q),\quad\text{ in the pointed measured Gromov--Hausdorff sense.}
\]
Namely, for every \(\varepsilon\in(0,1)\) and \(\mathcal L^1\)-a.e.\ \(R>1\), there exist \(\bar k\in\N\) and a sequence \((\psi^k)_{k\geq\bar k}\)
of Borel mappings \(\psi^k\colon B_{R r_k}(p)\to\Y\) such that the following properties are verified:
\begin{itemize}
\item[\(\rm i)\)] \(\psi^k(p)=q\),
\item[\(\rm ii)\)] \(\big|\sfd(x,y)-r_k\,\sfd_\Y\big(\psi^k(x),\psi^k(y)\big)\big|\leq\varepsilon r_k\) holds for every \(x,y\in B_{R r_k}(p)\),
\item[\(\rm iii)\)] \(B_{R-\varepsilon}(q)\) is contained in the open \(\varepsilon\)-neighbourhood of \(\psi^k\big(B_{R r_k}(p)\big)\),
\item[\(\rm iv)\)] \(\mm\big(B_{r_k}(p)\big)^{-1}\psi^k_\#\big(\mm|_{B_{R r_k}(p)}\big)\rightharpoonup\mm_\Y|_{B_R(q)}\) as \(k\to\infty\)
in duality with the space of bounded continuous functions \(f\colon\Y\to\R\) having bounded support.
\end{itemize}
\end{definition}

When we say that \(\Tan_p(\X,\sfd,\mm)\) contains a \textbf{unique} element, we mean that all its elements are isomorphic to each other in the following sense:
two given pointed metric measure spaces \((\Y_1,\sfd_{\Y_1},\mm_{\Y_1},q_1)\), \((\Y_2,\sfd_{\Y_2},\mm_{\Y_2},q_2)\) are said to be \textbf{isomorphic}
provided there exists an isometric bijection \(i\colon\Y_1\to\Y_2\) such that \(i(q_1)=q_2\) and \(i_\#\mm_{\Y_1}=\mm_{\Y_2}\).
This notion of isomorphism of pointed metric measure spaces is quite unnatural, as one would like to require that \(i\) is an isometric
bijection only between the supports of \(\mm_{\Y_1}\) and \(\mm_{\Y_2}\), but in general this is not allowed when working with the pointed measured
Gromov--Hausdorff topology, where `the whole space matters'. Nevertheless, this is not really an issue when (as in the present paper) only fully-supported
measures are taken into consideration.
\begin{remark}\label{rmk:equiv_tan_cone}{\rm
As proved in \cite[Proposition 3.28]{GigliMondinoSavare13}, a given pointed metric measure space \((\Y,\sfd_\Y,\mm_\Y,q)\) belongs to \(\Tan_p(\X,\sfd,\mm)\) 
if and only if there exist \(r_k\searrow 0\), \(R_k\nearrow\infty\), \(\varepsilon_k\searrow 0\), and Borel mappings \(\psi^k\colon B_{R_k r_k}(p)\to\Y\) such
that the following properties are verified:
\begin{itemize}
\item[\(\rm i')\)] \(\psi^k(p)=q\),
\item[\(\rm ii')\)] \(\big|\sfd(x,y)-r_k\,\sfd_\Y\big(\psi^k(x),\psi^k(y)\big)\big|\leq\varepsilon_k r_k\) holds for every \(x,y\in B_{R_k r_k}(p)\),
\item[\(\rm iii')\)] \(B_{R_k-\varepsilon_k}(q)\) is contained in the open \(\varepsilon_k\)-neighbourhood of \(\psi^k\big(B_{R_k r_k}(p)\big)\),
\item[\(\rm iv')\)] \(\mm\big(B_{r_k}(p)\big)^{-1}\psi^k_\#\big(\mm|_{B_{R_k r_k}(p)}\big)\rightharpoonup\mm_\Y\) as \(k\to\infty\)
in duality with the space of bounded continuous functions \(f\colon\Y\to\R\) having bounded support.
\fr
\end{itemize}
}\end{remark}

Notice that if \((\X,\sfd,\mm)\) is \(C\)-doubling, then \((\X,\sfd/r,\mm^r_x,x)\) is \(C\)-doubling for every \(x\in\X\) and \(r>0\).
Thanks to this observation, we deduce that, by combining \cite[Proposition 3.33]{GigliMondinoSavare13} with \cite[Proposition 6.3]{GP16},
one can readily obtain the following result:
\begin{lemma}\label{lem:local_tan}
Let \((\X,\sfd,\mm)\) be a doubling metric measure space and \(E\subset\X\) a Borel set. Then
\[
\Tan_x(\X,\sfd,\mm)=\Tan_x\big(E,\sfd|_{E\times E},\mm|_E\big),\quad\text{ for }\mm\text{-a.e.\ }x\in E.
\]
\end{lemma}
\subsection{Metric differential}
Let us briefly recall the concept of \textbf{metric differential}, introduced by Kirchheim in \cite{Kirchheim94}.
Let \((\X,\sfd)\) be a metric space, \(E\subset\R^n\) a Borel set, and \(f\colon E\to\X\) a Lipschitz map.
Being \(f(E)\) separable, we can find an isometric embedding \(\iota\colon f(E)\to\ell^\infty\). Fix any Lipschitz
extension \(\bar f\colon\R^n\to\ell^\infty\) of \(\iota\circ f\colon E\to\ell^\infty\). Then for \(\mathcal L^n\)-a.e.\ \(x\in E\) the limit
\[
{\rm md}_x(f)(v)\coloneqq\lim_{r\searrow 0}\frac{\big\|\bar f(x+rv)-\bar f(x)\big\|_{\ell^\infty}}{r}
\]
exists and is finite for every \(v\in\R^n\). Moreover, the resulting function \({\rm md}_x(f)\colon\R^n\to[0,+\infty)\) is a seminorm on \(\R^n\),
and is independent of the chosen extension \(\bar f\), for \(\mathcal L^n\)-a.e.\ point \(x\in E\). When \(f\) is biLipschitz with its image,
\({\rm md}_x(f)\) is a norm for \(\mathcal L^n\)-a.e.\ \(x\in E\). One also has that 
\begin{equation}\label{eq:md_prop_pre}
\lim_{\R^n\ni y\to x}\frac{\big\|\bar f(y)-\bar f(x)\big\|_{\ell^\infty}-{\rm md}_x(f)(y-x)}{|y-x|}=0,\quad\text{ for }\mathcal L^n\text{-a.e.\ }x\in E,
\end{equation}
as proved in \cite[Theorem 2]{Kirchheim94}. We will actually need a consequence of \eqref{eq:md_prop}, which we are going to discuss below. Before passing
to its statement, we fix some additional terminology.
\smallskip

The set \({\sf sn}_n\) of all seminorms on \(\R^n\) is a complete, separable metric space if endowed with the distance \({\sf D}_n\), which is given by
\[
{\sf D}_n({\sf n}_1,{\sf n}_2)\coloneqq\underset{\substack{v\in\R^n:\\|v|\leq 1}}\sup\big|{\sf n}_1(v)-{\sf n}_2(v)\big|,
\quad\text{ for every }{\sf n}_1,{\sf n}_2\in{\sf sn}_n.
\]
Then \(E\ni x\mapsto{\rm md}_x(f)\in{\sf sn}_n\) is Borel measurable, as it was pointed out in \cite[Theorem 3.1]{GigliTyulenev21}.
\begin{lemma}\label{lem:md_prop}
Let \((\X,\sfd)\) be a metric space. Let \(f\colon E\to\X\) be a Lipschitz map, for some Borel set \(E\subset\R^n\). Then there exists a partition
\((K_j)_{j\in\N}\) of \(E\) (up to \(\mathcal L^n\)-null sets) into compact sets with the following property: given any \(j\in\N\), it holds that
\begin{equation}\label{eq:md_prop}
\lim_{K_j\ni y,z\to x}\frac{\sfd\big(f(y),f(z)\big)-{\rm md}_x(f)(y-z)}{|y-z|}=0,\quad\text{ for }\mathcal L^n\text{-a.e.\ }x\in K_j.
\end{equation}
\end{lemma}
\begin{proof}
The property \eqref{eq:md_prop_pre} can be equivalently rephrased by saying that \(\phi_i\searrow 0\) holds \(\mathcal L^n\)-a.e.\ on \(E\) as \(i\to\infty\),
where for every \(i\in\N\) we define
\[
\phi_i(x)\coloneqq\sup_{y\in B_{1/i}^{|\cdot|}(x)\setminus\{x\}}\frac{\Big|\big\|\bar f(y)-\bar f(x)\big\|_{\ell^\infty}-{\rm md}_x(f)(y-x)\Big|}{|y-x|},
\quad\text{ for }\mathcal L^n\text{-a.e.\ }x\in E.
\]
By applying Lusin Theorem to \(E\ni x\mapsto{\rm md}_x(f)\in{\sf sn}_n\) and Egorov Theorem to \((\phi_i)_{i\in\N}\), we obtain a sequence
\((K_j)_{j\in\N}\) of pairwise disjoint, compact subsets of \(E\) with \(\mathcal L^n\big(E\setminus\bigcup_{j\in\N}K_j\big)=0\) such that
\(K_j\ni x\mapsto{\rm md}_x(f)\) is continuous and \(\phi_i|_{K_j}\rightrightarrows 0\) uniformly as \(i\to\infty\) for any \(j\in\N\).
Therefore, given any \(j\in\N\), \(x\in K_j\), and \(\varepsilon>0\), we can find an index \(i\in\N\) such that \(\phi_i(y)\leq\varepsilon\)
and \({\sf D}_n\big({\rm md}_y(f),{\rm md}_x(f)\big)\leq\varepsilon\) for every \(y\in B_{1/i}^{|\cdot|}(x)\cap K_j\), whence it follows that
\[\begin{split}
\frac{\big|\sfd\big(f(y),f(z)\big)-{\rm md}_x(f)(y-z)\big|}{|y-z|}&\leq\phi_i(y)+{\sf D}_n\big({\rm md}_y(f),{\rm md}_x(f)\big)\leq 2\varepsilon
\end{split}\]
holds for every \(y,z\in B_{1/(2i)}^{|\cdot|}(x)\cap K_j\) with \(y\neq z\). This gives \eqref{eq:md_prop}, as desired.
\end{proof}
\subsection{Essentially rectifiable spaces}
Let \((\X,\sfd,\mm)\) be a metric measure space. Then we say that a couple \((U,\varphi)\) is an \textbf{\(n\)-chart} on \((\X,\sfd,\mm)\), for some \(n\in\N\),
provided \(U\subset\X\) is a Borel set such that \(\mm|_U\ll\mathcal H^n\) and \(\varphi\colon U\to\R^n\) is a mapping which is biLipschitz with its image.
Following \cite{GP16,IPS21}, we say that \((\X,\sfd,\mm)\) is \textbf{\(\mm\)-rectifiable} provided it admits an \textbf{atlas}, \emph{i.e.}, a countable family
\(\mathscr A=\big\{(U_i,\varphi_i)\big\}_{i\in\N}\) of \(n_i\)-charts \(\varphi_i\colon U_i\to\R^{n_i}\) on \((\X,\sfd,\mm)\) (for some \(n_i\in\N\))
such that \(\{U_i\}_{i\in\N}\) is a Borel partition of \(\X\) up to \(\mm\)-null sets. Notice that we do not assume that \(\sup_{i\in\N}n_i<+\infty\).
We define \(n\colon\X\to\N\) as \(n(x)\coloneqq 0\) for every \(x\in\X\setminus\bigcup_{i\in\N}U_i\) and
\[
n(x)\coloneqq n_i,\quad\text{ for every }i\in\N\text{ and }x\in U_i.
\]
It can be readily checked that the function \(n\) is \(\mm\)-a.e.\ independent of the chosen atlas \(\mathscr A\).

Given any \(i\in\N\) and \(\mm\)-a.e.\ \(x\in U_i\), we define the norm \(\|\cdot\|_x\colon\R^{n(x)}\to[0,+\infty)\) on \(\R^{n(x)}\) as
\begin{equation}\label{eq:def_ptwse_norm_rect}
\|v\|_x\coloneqq{\rm md}_{\varphi_i(x)}(\varphi_i^{-1})(v),\quad\text{ for every }v\in\R^{n(x)}.
\end{equation}
The fact that \((\varphi_i)_\#(\mm|_{U_i})\ll\mathcal L^{n(x)}\) ensures that \(\|\cdot\|_x\) is \(\mm\)-a.e.\ independent of the atlas \(\mathscr A\).

We denote by \(\mathcal H^{n(x)}_x\) the \(n(x)\)-dimensional Hausdorff measure on \(\big(\R^{n(x)},\|\cdot\|_x\big)\) and by
\[
\underline{\mathcal H}^{n(x)}_x\coloneqq\frac{\mathcal H^{n(x)}_x}{\mathcal H^{n(x)}_x\big(B_1^{\|\cdot\|_x}(0)\big)}
\]
its normalisation. Moreover, for any \(i\in\N\) we can find a Borel function \(\theta_i\colon U_i\to[0,+\infty)\) such that
\(\mm|_{U_i}=\theta_i\mathcal H^{n_i}_\sfd|_{U_i}\). We define the density function \(\theta\colon\X\to[0,+\infty)\) as
\(\theta\coloneqq\sum_{i\in\N}\nchi_{U_i}\theta_i\).

Lemma \ref{lem:md_prop} implies that, up to refining the atlas \(\mathscr A\), it is not restrictive to assume that
\begin{equation}\label{eq:norm_x_prop}
\lim_{U_i\ni y,z\to x}\frac{\big|\sfd(y,z)-\|\varphi_i(y)-\varphi_i(z)\|_x\big|}{\sfd(y,z)}=0,
\quad\text{ for every }i\in\N\text{ and }\mm\text{-a.e.\ }x\in U_i.
\end{equation}
\section{Proof of Theorem \ref{thm:generic_tan_PI_rect}}\label{s:proof_thm3}
Theorem \ref{thm:generic_tan_PI_rect} is a consequence of the following two results, of independent interest.
\begin{proposition}\label{prop:equiv_iH_rect}
Let \((\X,\sfd,\mm)\) be an \(\mm\)-rectifiable space. If \(\|\cdot\|_x\) is a Hilbert norm on \(\R^{n(x)}\) for \(\mm\)-a.e.\ point \(x\in\X\),
then \((\X,\sfd,\mm)\) is infinitesimally Hilbertian. In the case where \((\X,\sfd,\mm)\) is also a PI space, the converse implication is
verified as well.
\end{proposition}

\begin{proof}
The first part of the statement follows from \cite[Lemma 4.1]{IPS21}, \cite[Theorem 1.2]{IPS21}, and \cite[Proposition 2.3.17]{Gigli14},
whereas the last part can be obtained by taking also \cite[Theorem 1.3]{IPS21} and the results of \cite{Cheeger00} into account.
Alternatively, the last part of the statement can be deduced from \cite[Corollary 6.7]{ErikssonBiqueSoultanis21}.
\end{proof}
\begin{proposition}\label{prop:tan_to_rect}
Let \((\X,\sfd,\mm)\) be a doubling, \(\mm\)-rectifiable space. Then for \(\mm\)-a.e.\ \(x\in\X\) the tangent cone
\(\Tan_x(\X,\sfd,\mm)\) consists uniquely of the space \(\big(\R^{n(x)},\|\cdot\|_x,\underline{\mathcal H}^{n(x)}_x,0\big)\).
\end{proposition}
\begin{proof}
Let \(\big\{(U_i,\varphi_i)\big\}_{i\in\N}\) be an atlas of \((\X,\sfd,\mm)\). An application of Lusin Theorem yields the existence
of a partition \((K^i_j)_{j\in\N}\) of \(U_i\) (up to \(\mm\)-null sets) into compact sets such that each \(\theta|_{K^i_j}\) is continuous.
Moreover, Lemma \ref{lem:local_tan} gives \(\Tan_x(\X,\sfd,\mm)=\Tan_x(K^i_j,\sfd,\mm)\) for \(\mm\text{-a.e.\ }x\in K^i_j\).
Hence, we can assume without loss of generality that \(\X\) is compact, that \(\mm=\theta\mathcal H_\sfd^n\) for some continuous density
\(\theta\colon\X\to[0,+\infty)\), and that there exists a mapping \(\varphi\colon\X\to\R^n\) which is biLipschitz with its image. Then
our aim is to show that for \(\mm\)-a.e.\ \(x\in\X\) the pointed metric measure space \(\big(\R^n,\|\cdot\|_x,\underline{\mathcal H}^n_x,0\big)\)
is the unique element of the tangent cone \(\Tan_x(\X,\sfd,\mm)\).

Let \(x\in\X\) be a given point where \eqref{eq:norm_x_prop} holds  and \(\theta(x)>0\) (this property holds \(\mm\)-a.e.).
Fix any \(r_k\searrow 0\) and \(0<\varepsilon<1<R\). Then \eqref{eq:md_prop} yields a sequence \(\delta_k\searrow 0\) such that \(2\delta_k R<\varepsilon\),
\begin{equation}\label{eq:tan_to_rect_1}
\big|\sfd(y,z)-\|\varphi(y)-\varphi(z)\|_x\big|\leq\delta_k\sfd(y,z),\quad\text{ for every }y,z\in B_{R r_k}^\sfd(x),
\end{equation}
and \(|\theta(y)-\theta(x)|\leq\delta_k\theta(x)\) for every \(y\in B_{R r_k}^\sfd(x)\). Define
\(\psi^k(y)\coloneqq\frac{\varphi(y)-\varphi(x)}{r_k}\) for all \(y\in B_{R r_k}^\sfd(x)\).
Then the Borel maps \(\psi^k\colon B_{R r_k}^\sfd(x)\to\big(\R^n,\|\cdot\|_x,\underline{\mathcal H}^n_x,0\big)\) verify
the conditions in Definition \ref{def:tan_cone}:\\
{\color{blue}\(\rm i)\)} \(\psi^k(x)=0\) by definition.\\
{\color{blue}\(\rm ii)\)} It follows from \eqref{eq:tan_to_rect_1} that
\[
\big|\sfd(y,z)-r_k\|\psi^k(y)-\psi^k(z)\|_x\big|\leq\delta_k\sfd(y,z)\leq\varepsilon r_k,
\quad\text{ for every }y,z\in B_{R r_k}^\sfd(x).
\]
{\color{blue}\(\rm iii)\)} The same estimates also show that \(\psi^k\colon\big(B_{R r_k}^\sfd(x),\sfd\big)\to\big(\R^n,r_k\|\cdot\|_x\big)\)
is \(L_k\)-biLipschitz with its image, where we set \(L_k\coloneqq 1+\delta_k\). In particular, we obtain that
\begin{equation}\label{eq:tan_to_rect_2}
B_{R/L_k}^{\|\cdot\|_x}(0)=B_{R r_k/L_k}^{r_k\|\cdot\|_x}(0)\subset\psi^k\big(B_{R r_k}^\sfd(x)\big)\subset
B_{R r_k L_k}^{r_k\|\cdot\|_x}(0)=B_{R L_k}^{\|\cdot\|_x}(0).
\end{equation}
Given that \(R-\varepsilon<R/L_k\), we deduce from \eqref{eq:tan_to_rect_2} that \(B_{R-\varepsilon}^{\|\cdot\|_x}(0)\subset\psi^k\big(B_{R r_k}^\sfd(x)\big)\).\\
{\color{blue}\(\rm iv)\)} The \(L_k\)-biLipschitzianity of the mapping \(\psi^k|_{B_{R r_k}^\sfd(x)}\) also ensures that
\begin{equation}\label{eq:tan_to_rect_3}
\frac{r_k^n}{L_k^n}\mathcal H^n_x|_{\psi^k(B_{R r_k}^\sfd(x))}
\leq\psi^k_\#\big(\mathcal H^n_\sfd|_{B_{R r_k}^\sfd(x)}\big)\leq
r_k^n L_k^n\,\mathcal H^n_x|_{\psi^k(B_{R r_k}^\sfd(x))}.
\end{equation}
Recalling that \(\theta(x)(1-\delta_k)\leq\theta\leq\theta(x)L_k\) on \(B_{R r_k}^\sfd(x)\), we deduce from \eqref{eq:tan_to_rect_2} and
\eqref{eq:tan_to_rect_3} that
\[
\frac{\psi^k_\#\big(\mm|_{B_{R r_k}^\sfd(x)}\big)}{\mm\big(B_{r_k}^\sfd(x)\big)}\leq
\frac{L_k}{1-\delta_k}\frac{\psi^k_\#\big(\mathcal H^n_\sfd|_{B_{R r_k}^\sfd(x)}\big)}{\mathcal H^n_\sfd\big(B_{r_k}^\sfd(x)\big)}\leq
\frac{L_k^{2n+1}}{1-\delta_k}\frac{\mathcal H^n_x|_{B_{R L_k}^{\|\cdot\|_x}(0)}}{\mathcal H^n_x\big(B_{1/L_k}^{\|\cdot\|_x}(0)\big)}=
\frac{L_k^{3n+1}}{1-\delta_k}\,\underline{\mathcal H}^n_x|_{B_{R L_k}^{\|\cdot\|_x}(0)}.
\]
Similarly, we can estimate
\[
\frac{\psi^k_\#\big(\mm|_{B_{R r_k}^\sfd(x)}\big)}{\mm\big(B_{r_k}^\sfd(x)\big)}\geq
\frac{1-\delta_k}{L_k^{3n+1}}\,\underline{\mathcal H}^n_x|_{B_{R/L_k}^{\|\cdot\|_x}(0)}.
\]
Since \(L_k\to 1\) as \(k\to\infty\), we finally conclude that
\(\mm\big(B_{r_k}^\sfd(x)\big)^{-1}\psi^k_\#\big(\mm|_{B_{R r_k}^\sfd(x)}\big)\rightharpoonup\underline{\mathcal H}^n_x|_{B_R^{\|\cdot\|_x}(0)}\)
in duality with bounded continuous functions \(f\colon\R^n\to\R\) having bounded support.
\end{proof}
\begin{proof}[Proof of Theorem \ref{thm:generic_tan_PI_rect}]
Let \((\X,\sfd,\mm)\) be an infinitesimally Hilbertian, \(\mm\)-rectifiable PI space. The last part of Proposition \ref{prop:equiv_iH_rect}
tells that \(\|\cdot\|_x\) is a Hilbert norm for \(\mm\)-a.e.\ \(x\in\X\). Hence, Proposition \ref{prop:tan_to_rect} ensures that
for \(\mm\)-a.e.\ \(x\in\X\) the tangent cone \(\Tan_x(\X,\sfd,\mm)\) contains only the infinitesimally Hilbertian space
\(\big(\R^{n(x)},\|\cdot\|_x,\underline{\mathcal H}^{n(x)}_x,0\big)\), yielding the sought conclusion.
\end{proof}
\section{Proof of Theorem \ref{thm:example_non_PI}}\label{s:proof_thm1}
Let \(\X\subset\R^2\) be given by \(\X\coloneqq C\times C\), where \(C\subset\R\) is a Cantor set of positive \(\mathcal L^1\)-measure.
We endow \(\X\) with the distance \(\sfd\), given by \(\sfd(a,b)\coloneqq\|a-b\|_1\) for every \(a,b\in\X\), and with the measure \(\mm\coloneqq\mathcal L^2|_\X\).
\begin{proof}[Proof of Theorem \ref{thm:example_non_PI}]
We check that \((\X,\sfd,\mm)\) verifies Theorem \ref{thm:example_non_PI}. It is easy to show that it is \(2\)-Ahlfors regular and \(\mm\)-rectifiable.
Moreover, the space \(\X\) (being totally disconnected) cannot contain non-constant absolutely continuous curves, thus the equivalent characterisations
of \(H^{1,2}(\X)\) in \cite{AmbrosioGigliSavare11-3} imply that \(H^{1,2}(\X)=L^2(\X)\) and \(\|f\|_{H^{1,2}(\X)}=\|f\|_{L^2(\X)}\) for all
\(f\in H^{1,2}(\X)\). Hence, trivially, the metric measure space \((\X,\sfd,\mm)\) is infinitesimally Hilbertian. Finally, it follows from Lemma \ref{lem:local_tan}
that \(\Tan_a(\X,\sfd,\mm)=\Tan_a\big(\R^2,\|\cdot\|_1,\mathcal L^2\big)\) holds for \(\mm\)-a.e.\ point \(a\in\X\), and it is immediate to check that the norms
\(\big\{\|\cdot\|_a\big\}_{a\in\R^2}\) associated with the \(\mathcal L^2\)-rectifiable space \(\big(\R^2,\|\cdot\|_1,\mathcal L^2\big)\)
satisfy \(\|\cdot\|_a=\|\cdot\|_1\) for every \(a\in\R^2\). This fact implies (thanks to Proposition \ref{prop:tan_to_rect})
that for \(\mm\)-a.e.\ \(a\in\X\) the tangent cone \(\Tan_a(\X,\sfd,\mm)\) consists exclusively of the space
\(\big(\R^2,\|\cdot\|_1,\underline{\mathcal H}^2_{\|\cdot\|_1},0\big)\), which is not infinitesimally Hilbertian by Proposition \ref{prop:equiv_iH_rect}.
\end{proof}
\begin{remark}{\rm
It is also possible to provide an example of metric measure space \((\X,\sfd,\mm)\) verifying Theorem \ref{thm:example_non_PI} whose Sobolev space
\(H^{1,2}(\X)\) is non-trivial. To this aim, fix a Cantor set \(C\subset\R\) of positive \(\mathcal L^1\)-measure and define \(\X\coloneqq C\times\R\).
We endow the space \(\X\subset\R^2\) with the distance \(\sfd(a,b)\coloneqq\|a-b\|_1\) and with the measure \(\mm\coloneqq\mathcal L^2|_\X\). Exactly as before,
\((\X,\sfd,\mm)\) is \(2\)-Ahlfors regular, \(\mm\)-rectifiable, and its tangents are \(\mm\)-a.e.\ unique and infinitesimally non-Hilbertian. The infinitesimal
Hilbertianity of \((\X,\sfd,\mm)\) boils down to the fact that all norms on \(\R\) are Hilbert. Indeed, one can check
that a given function \(f\in L^2(\X)\) belongs to \(H^{1,2}(\X)\) if and only if \(f(x,\cdot)\in W^{1,2}(\R)\) holds for \(\mathcal L^1\)-a.e.\ \(x\in C\)
and \(\int_C\big\||Df(x,\cdot)|\big\|_{L^2(\R)}^2\,\d\mathcal L^1(x)<+\infty\). Moreover, for any function \(f\in H^{1,2}(\X)\) we have that
\[
\|f\|_{H^{1,2}(\X)}^2=\int|f|^2\,\d\mm+\int_C\big\||Df(x,\cdot)|\big\|_{L^2(\R)}^2\,\d\mathcal L^1(x).
\]
In particular, \(H^{1,2}(\X)\) is a Hilbert space, thus yielding the sought conclusion.
\fr}\end{remark}
\section{Proof of Theorem \ref{thm:example_PI}}\label{s:proof_thm2}
By a \emph{dyadic square} in the plane we mean an open square \(Q\subset\R^2\) of the form
\[
Q=Q^k_{i,j}\coloneqq\big(i 2^k,(i+1)2^k\big)\times\big(j 2^k,(j+1)2^k\big),\quad\text{ for some }i,j,k\in\Z.
\]
We denote by \(\mathcal D\) the family of all dyadic squares in the plane. The side-length of a dyadic square \(Q\in\mathcal D\)
is denoted by \(\ell(Q)\). Consider the family \(\mathcal W\coloneqq\{Q^k_{i,j}\,:\,k\in\Z,\,(i,j)\in F\}\), where
\[\begin{split}
F\coloneqq\big\{&(1,0),(1,1),(0,1),(-1,1),(-2,1),(-2,0),\\&(-2,-1),(-2,-2),(-1,-2),(0,-2),(1,-2),(1,-1)\big\}.
\end{split}\]
Observe that \(\mathcal W\) is the Whitney decomposition of \(\R^2\setminus\{0\}\). Given any \(Q\in\mathcal W\) with
\(\ell(Q)=2^k\), we define the family \(\mathcal S(Q)\subset\mathcal D\) as
\[
\mathcal S(Q)\coloneqq\big\{Q'\in\mathcal D\;\big|\;Q'\subset Q,\,\ell(Q')=2^{k+\min\{k,0\}}\big\}.
\]
It holds that \(\mathcal S(Q)=\{Q\}\) if \(k\geq 0\), while \(\mathcal S(Q)\) is a collection of \(4^{-k}\) pairwise disjoint dyadic squares
of side-length \(4^k\) if \(k<0\). It also holds \(\bar Q=\bigcup_{Q'\in\mathcal S(Q)}\bar Q'\). Define
\(\mathcal S\coloneqq\bigcup_{Q\in\mathcal W}\mathcal S(Q)\). Moreover, we define \(N_k\coloneqq[-2^{-k},2^{-k}]^2\subset\R^2\) and
\(\mathcal S_k\coloneqq\{Q\in\mathcal S\,:\,Q\subset N_k\}\) for every \(k\in\N\). Observe that \(N_k=\bigcup_{Q\in\mathcal S_k}\bar Q\) and
\(\ell(Q)\leq 4^{-(k+1)}\) for every \(Q\in\mathcal S_k\).
\begin{remark}\label{rmk:local_dist}{\rm
Given any function \(\rho\colon\R^2\to[1,2]\), it holds that
\[
\sfd_\rho^{N_k}(x,y)=\sfd_\rho(x,y),\quad\text{ for every }k\in\N\text{ and }x,y\in N_{k+2}.
\]
Indeed, the \(\sfd_\rho\)-distance between any two points in \(N_{k+2}\) cannot exceed \(2\sqrt 2/2^{k+1}\), while any Lipschitz curve \(\gamma\)
in \(\R^2\) which joins two points in \(N_{k+2}\) and intersects \(\R^2\setminus N_k\) satisfies the estimate \(\ell_\rho(\gamma)\geq 2(2^{-(k+1)}+2^{-(k+2)})\).
Given that \(2\big(\frac{1}{2^{k+1}}+\frac{1}{2^{k+2}}\big)=\frac{3}{2^{k+1}}>\frac{2\sqrt 2}{2^{k+1}}\), we deduce that to compute the \(\sfd_\rho\)-distance
between two points in \(N_{k+2}\) it is sufficient to consider just those Lipschitz curves which are contained in \(N_k\), whence the claimed identity follows.
\fr}\end{remark}
Given any \(n\in\N\), let us fix a smooth function \(\psi_n\colon(-1,2)^2\to[1,2]\) such that \(\psi_n=1\) on some neighbourhood
of \(\partial([0,1]^2)\) and \(\psi_n=2\) in the smaller square \([2^{-(n+2)},1-2^{-(n+2)}]^2\). We can further require that
\(\psi_n\leq\psi_{n+1}\) for every \(n\in\N\). Moreover, we define \(\psi_\infty\colon[0,1]^2\to\{1,2\}\) as
\(\psi_\infty\coloneqq\nchi_{\partial([0,1]^2)}+2\nchi_{(0,1)^2}\). Notice that \(\psi_n\nearrow\psi_\infty\) on \([0,1]^2\) as \(n\to\infty\).
For any \(Q\in\mathcal S\), we define the transformation \(\theta_Q\colon[0,1]^2\to\bar Q\) as
\(\theta_Q(x,y)\coloneqq(\tau_Q\circ\delta_{\ell(Q)})(x,y)\) for all \((x,y)\in[0,1]^2\), where \(\delta_\lambda\colon\R^2\to\R^2\)
is the dilation \((x,y)\mapsto(\lambda x,\lambda y)\), while \(\tau_Q\colon\R^2\to\R^2\) stands for the unique translation satisfying
\(\tau_Q([0,\ell(Q)]^2)=Q\). Given any \(k,n\in\N\), we define \(\rho^k_n\colon N_k\to[1,2]\) as
\[
\rho^k_n\coloneqq\nchi_{R\cap N_k}+\sum_{Q\in\mathcal S_k}\nchi_Q\,\psi_n\circ\theta_Q^{-1},
\]
where we set \(R\coloneqq\R^2\setminus\bigcup_{Q\in\mathcal S}Q\). Furthermore, we define the function \(\rho_\infty\colon\R^2\to\{1,2\}\) as
\[
\rho_\infty\coloneqq\nchi_R+2\nchi_{\R^2\setminus R}=\nchi_R+\sum_{Q\in\mathcal S}\nchi_Q\,\psi_\infty\circ\theta_Q^{-1}.
\]
We have that \(\rho^k_n\nearrow\rho_\infty\) on \(N_k\) as \(n\to\infty\), whence it follows that
\(\sfd^{N_k}_{\rho^k_n}\nearrow\sfd^{N_k}_{\rho_\infty}\) on \(N_k\times N_k\) as \(n\to\infty\). Given that \(\sfd_{\rho_\infty}\leq 2\,\sfd_{\rm Eucl}\),
the function \(\sfd^{N_k}_{\rho_\infty}\) is continuous on \(N_k\times N_k\) and thus accordingly \(\sfd^{N_k}_{\rho^k_n}\rightrightarrows\sfd^{N_k}_{\rho_\infty}\)
uniformly on the compact set \(N_k\times N_k\) as \(n\to\infty\). Therefore, we can choose \(n(k)\in\N\) so that
\(\sfd_{\rho_k}^{N_k}(a,b)\geq\sfd_{\rho_\infty}^{N_k}(a,b)-4^{-(k+2)}\geq\sfd_{\rho_\infty}(a,b)-4^{-(k+2)}\) for all \(a,b\in N_k\),
where we set \(\rho_k\coloneqq\rho_{n(k)}^k\). We can assume without loss of generality that \(\N\ni k\mapsto n(k)\in\N\) is strictly increasing.
We now define the auxiliary function \(m\colon\mathcal S\to\N\) as
\[
m(Q)\coloneqq\left\{\begin{array}{ll}
n(k),\\
0,
\end{array}\quad\begin{array}{ll}
\text{ if }Q\in\mathcal S_k\setminus\mathcal S_{k+1}\text{ for some }k\in\N,\\
\text{ if }Q\in\mathcal S\setminus\mathcal S_0.
\end{array}\right.
\]
Finally, we define the function \(\rho\colon\R^2\to[1,2]\) as
\[
\rho\coloneqq\nchi_R+\sum_{Q\in\mathcal S}\nchi_Q\,\psi_{m(Q)}\circ\theta_Q^{-1}.
\]
Observe that \(\rho\) is smooth on \(\R^2\setminus\{0\}\).
Given that \(\rho\geq\rho_k\) on \(N_k\) for any \(k\in\N\) by construction, we deduce that \(\sfd_\rho^{N_k}\geq\sfd_{\rho_k}^{N_k}\) and thus
\begin{equation}\label{eq:aux_lower_bound_d_rho}
\sfd_\rho(a,b)\geq\sfd_{\rho_\infty}(a,b)-\frac{1}{4^{k+2}},\quad\text{ for every }k\in\N\text{ and }a,b\in N_{k+2},
\end{equation}
where we used that \(\sfd_\rho=\sfd_\rho^{N_k}\geq\sfd_{\rho_k}^{N_k}\geq\sfd_{\rho_\infty}-4^{-(k+2)}\)
on \(N_{k+2}\times N_{k+2}\) by Remark \ref{rmk:local_dist}.
\begin{lemma}\label{lem:bound_d_rho}
Let \(k\geq 2\) be given. Then it holds that
\begin{equation}\label{eq:bound_d_rho}
\|a-b\|_1-\frac{1}{4^k}\leq\sfd_\rho(a,b)\leq\|a-b\|_1+\frac{1}{4^k},\quad\text{ for every }a,b\in N_k.
\end{equation}
\end{lemma}
\begin{proof}
By continuity, it is sufficient to check the validity of the statement when \(a,b\in N_k\setminus R\).\\
{\color{blue}\textsc{Upper bound.}} Call \(Q_a\) (resp.\ \(Q_b\)) the unique element of \(\mathcal S_k\) containing \(a\) (resp.\ \(b\)).
We can find two points \(a'\in\partial Q_a\) and \(b'\in\partial Q_b\) such that \(\|a'-b'\|_1\leq\|a-b\|_1\). We can also require that
each of the segments \([a,a']\) and \([b',b]\) is either horizontal or vertical. Hence, calling \(\gamma_a\) (resp.\ \(\gamma_b\)) the
constant-speed parametrisation of the interval \([a,a']\) (resp.\ of \([b',b]\)), we have that \(\ell_\rho(\gamma_b)\leq 2\ell(Q_a)\)
and \(\ell_\rho(\gamma_a)\leq 2\ell(Q_b)\). We can construct a polygonal curve \(\tilde\gamma\colon[0,1]\to R\) with
\(\tilde\gamma_0=a'\), \(\tilde\gamma_1=b'\), and \(\ell_\rho(\tilde\gamma)=\|a'-b'\|_1\). Then the concatenation
\(\gamma\coloneqq\gamma_a*\tilde\gamma*\gamma_b\) satisfies
\[
\ell_\rho(\gamma)=\ell_\rho(\gamma_a)+\ell_\rho(\tilde\gamma)+\ell_\rho(\gamma_b)\leq 2\ell(Q_a)+\|a'-b'\|_1+2\ell(Q_b)\leq
\|a-b\|_1+\frac{1}{4^k}.
\]
Since the curve \(\gamma\) joins \(a\) and \(b\), we can conclude that the upper bound in \eqref{eq:bound_d_rho} is verified.\\
{\color{blue}\textsc{Lower bound.}} Fix any Lipschitz curve \(\gamma\colon[0,1]\to\R^2\) joining \(a\) and \(b\). We denote by \(H\subset\R^2\)
(resp.\ \(V\subset\R^2\)) the intersection between \(R\) and \(\R\times\{j 2^k\,:\,j,k\in\Z\}\) (resp.\ \(\{i 2^k\,:\,i,k\in\Z\}\times\R\)).
Notice that \(H\cap V\) is a countable family. We write \([0,1]=I_M\cup I_H\cup I_V\), where we define
\[
I_M\coloneqq\big\{t\in[0,1]\;\big|\;\gamma_t\in\R^2\setminus R\big\},\quad I_H\coloneqq\big\{t\in[0,1]\;\big|\;\gamma_t\in H\big\},
\quad I_V\coloneqq\big\{t\in[0,1]\;\big|\;\gamma_t\in V\big\}.
\]
Denote \(a=(a_1,a_2)\), \(b=(b_1,b_2)\), and \(\gamma=(\gamma^1,\gamma^2)\). Then \(\gamma^1\) is a Lipschitz curve in \(\R\) that joins \(a_1\) and \(b_1\), so
that \(|a_1-b_1|\leq\int_0^1|\dot\gamma^1_t|\,\d t\). For any \(i,k\in\Z\) it holds that \(\gamma^1_t=i 2^k\) for every \(t\in\gamma^{-1}(\{i 2^k\}\times\R)\)
and thus \(\dot\gamma^1_t=0\) for a.e.\ \(t\in\gamma^{-1}(\{i 2^k\}\times\R)\). This implies \(\dot\gamma^1_t=0\) for a.e.\ \(t\in I_V\), so that
\(|a_1-b_1|\leq\int_{I_M\cup I_H}|\dot\gamma^1_t|\,\d t\). Similarly, one has \(|a_2-b_2|\leq\int_{I_M\cup I_V}|\dot\gamma^2_t|\,\d t\). Therefore, we can estimate
\[\begin{split}
\ell_{\rho_\infty}(\gamma)&=\int_{I_M}2|\dot\gamma_t|\,\d t+\int_{I_H}|\dot\gamma_t|\,\d t+\int_{I_V}|\dot\gamma_t|\,\d t
\geq\int_{I_M}|\dot\gamma^1_t|+|\dot\gamma^2_t|\,\d t+\int_{I_H}|\dot\gamma^1_t|\,\d t+\int_{I_V}|\dot\gamma^2_t|\,\d t\\
&=\int_{I_M\cup I_H}|\dot\gamma^1_t|\,\d t+\int_{I_M\cup I_V}|\dot\gamma^2_t|\,\d t\geq|a_1-b_1|+|a_2-b_2|=\|a-b\|_1.
\end{split}\]
Thanks to the arbitrariness of \(\gamma\), we deduce that \(\sfd_{\rho_\infty}(a,b)\geq\|a-b\|_1\). Recalling \eqref{eq:aux_lower_bound_d_rho},
we can finally conclude that the lower bound in \eqref{eq:bound_d_rho} is verified, whence the statement follows.
\end{proof}
We endow the smooth manifold \(M\coloneqq\R^2\setminus\{0\}\) with the Riemannian metric \(g\), which is defined as \(g_x(v,w)\coloneqq\rho(x)\langle v,w\rangle\).
Call \(\mm\) the \(2\)-dimensional Hausdorff measure on \((\R^2,\sfd_\rho)\). Given that the restriction of \(\sfd_\rho\) to \(M\) is (by definition) the length
distance induced by the Riemannian metric \(g\), we have that \(\mm|_M\) coincides with the volume measure of \((M,g)\). By exploiting the fact that
\(\sfd_{\rm Eucl}\leq\sfd_\rho\leq 2\,\sfd_{\rm Eucl}\), one can also deduce that \(\mathcal L^2\leq\mm\leq 2\mathcal L^2\), thus in particular
\((\R^2,\sfd_\rho,\mm)\) is an \(\mm\)-rectifiable, \(2\)-Ahlfors regular PI space. Moreover, Lemma \ref{lem:d_rho_iH} ensures that the metric measure space
\((\R^2,\sfd_\rho,\mm)\) is infinitesimally Hilbertian.

\begin{proof}[Proof of Theorem \ref{thm:example_PI}]
The metric measure space \((\R^2,\sfd_\rho,\mm)\) constructed above satisfies the assumptions of Theorem \ref{thm:example_PI}. Let us verify that
\(\Tan_0(\R^2,\sfd_\rho,\mm)\) contains an infinitesimally non-Hilbertian element. Given any \(k\in\N\), we define
\(r_k\coloneqq 1/(k2^k)\), \(R_k\coloneqq k\), \(\varepsilon_k\coloneqq k/2^k\), and
\[
\psi^k(a)\coloneqq\frac{a}{r_k},\quad\text{ for every }a\in B_{R_k r_k}^{\sfd_\rho}(x).
\]
Let us check that the Borel maps \(\psi^k\colon B_{R_k r_k}^{\sfd_\rho}(0)\to\R^2\) satisfy the conditions in Remark \ref{rmk:equiv_tan_cone},
when the target \(\R^2\) is endowed with the norm \(\|\cdot\|_1\) and a suitable measure \(\mu\) with \(0\in\spt(\mu)\).\\
{\color{blue}\(\rm i')\)} By definition, \(\psi^k(0)=0\) for every \(k\in\N\).\\
{\color{blue}\(\rm ii')\)} Let \(k\geq 2\) be fixed. Since \(\sfd_{\rm Eucl}\leq\sfd_\rho\), we have that \(B_{R_k r_k}^{\sfd_\rho}(0)
=B_{2^{-k}}^{\sfd_\rho}(0)\subset N_k\). Therefore,

\[
\big|\sfd_\rho(a,b)-r_k\|\psi^k(a)-\psi^k(b)\|_1\big|=\big|\sfd_\rho(a,b)-\|a-b\|_1\big|\leq\frac{1}{4^k}=\varepsilon_k r_k
\]
holds for every \(a,b\in B_{R_k r_k}^{\sfd_\rho}(x)\), where the inequality follows from Lemma \ref{lem:bound_d_rho}.\\
{\color{blue}\(\rm iii')\)} Fix any \(k\geq 2\) and \(v\in B_{R_k-\varepsilon_k}^{\|\cdot\|_1}(0)\). Given that
\(\|r_k v\|_1<(R_k-\varepsilon_k)r_k=2^{-k}-4^{-k}<2^{-k}\), one has \(a\coloneqq r_k v\in B_{2^{-k}}^{\|\cdot\|_1}(0)\subset N_k\).
Hence, Lemma \ref{lem:bound_d_rho} ensures that \(\sfd_\rho(a,0)\leq\|a\|_1+4^{-k}<2^{-k}\), which implies that
\(a\in B_{2^{-k}}^{\sfd_\rho}(0)=B_{R_k r_k}^{\sfd_\rho}(0)\) and thus \(v=\psi^k(a)\in\psi^k\big(B_{R_k r_k}^{\sfd_\rho}(0)\big)\), as desired.\\
{\color{blue}\(\rm iv')\)} We aim to find a boundedly-finite Borel measure \(\mu\geq 0\) on \(\big(\R^2,\|\cdot\|_1\big)\) such that
\[
\mu_k\coloneqq\frac{\psi^k_\#\big(\mm|_{B_{R_k r_k}^{\sfd_\rho}(0)}\big)}{\mm\big(B_{r_k}^{\sfd_\rho}(0)\big)}\rightharpoonup\mu,
\quad\text{ in duality with compactly-supported, continuous functions,}
\]
up to a subsequence in \(k\). Up to a diagonalisation argument, it is sufficient to show that for any compact set \(K\subset\R^2\) the sequence
\(\mu_k|_K\) weakly subconverges to some finite Borel measure on \(K\) in duality with continuous functions on \(K\). In turn, to obtain the
latter condition it is enough to prove that \(\sup_{k\in\N}\mu_k(K)<+\infty\). Let us check it: for any \(k\in\N\), we can estimate
\[
\mu_k(K)=\frac{\mm\big((\psi^k)^{-1}(K)\cap B_{2^{-k}}^{\sfd_\rho}(0)\big)}{\mm\big(B_{r_k}^{\sfd_\rho}(0)\big)}\overset{(\star)}\leq
\frac{\mm(r_k K)}{\mm\big(B_{r_k/2}^{\|\cdot\|_2}(0)\big)}\leq\frac{2\mathcal L^2(r_k K)}{\mathcal L^2\big(B_{r_k/2}^{\|\cdot\|_2}(0)\big)}=\frac{8\mathcal L^2(K)}{\pi},
\]
where in the starred inequality we used the fact that \(\sfd_\rho\leq 2\,\sfd_{\rm Eucl}\) and thus \(B_{r_k/2}^{\|\cdot\|_2}(0)\subset B_{r_k}^{\sfd_\rho}(0)\).
Finally, we aim to show that \(0\in{\rm spt}(\mu)\), or equivalently that \(\limsup_{k\to\infty}\mu_k\big(B_\delta^{\|\cdot\|_2}(0)\big)>0\) holds for every
\(\delta\in(0,1)\). Given any such \(\delta\), we can find \(\bar k\in\N\) and \(C_\delta>0\) such that \(R_k/2>\delta\) and \(\mm\big(B_{\delta r_k}^{\|\cdot\|_2}(0)\big)
\geq C_\delta\,\mm\big(B_{r_k}^{\|\cdot\|_2}(0)\big)\) for every \(k\geq\bar k\); for the latter property, we are using the fact that \(\big(\R^2,\|\cdot\|_2,\mm\big)\)
is doubling. In particular, \(B_{\delta r_k}^{\|\cdot\|_2}(0)\subset B_{R_k r_k/2}^{\|\cdot\|_2}(0)\) for all \(k\geq\bar k\). Hence,
\[
\mu_k\big(B_\delta^{\|\cdot\|_2}(0)\big)=
\frac{\mm\big(B_{\delta r_k}^{\|\cdot\|_2}(0)\cap B_{R_k r_k}^{\sfd_\rho}(0)\big)}{\mm\big(B_{r_k}^{\sfd_\rho}(0)\big)}
\geq\frac{\mm\big(B_{\delta r_k}^{\|\cdot\|_2}(0)\cap B_{R_k r_k/2}^{\|\cdot\|_2}(0)\big)}{\mm\big(B_{r_k}^{\|\cdot\|_2}(0)\big)}\geq C_\delta,
\quad\text{ for all }k\geq\bar k.
\]
All in all, we proved that \(\big(\R^2,\|\cdot\|_1,\mu,0\big)\in\Tan_0(\R^2,\sfd_\rho,\mm)\). Since \(\|\cdot\|_1\) is a non-Hilbert norm, we conclude from
\cite[Lemma 4.4]{LP20} that \(\big(\R^2,\|\cdot\|_1,\mu\big)\) is not infinitesimally Hilbertian, thus completing the proof of Theorem \ref{thm:example_PI}.
\end{proof}
\begin{remark}{\rm
Theorem \ref{thm:example_PI} could be modified so that for any closed set \(F \subset \R^2\) of Lebesgue measure zero, there exists a distance \(\sfd_F\) on \(\R^2\) so that \((\R^2,\sfd_F,\mm)\) is an infinitesimally Hilbertian, \(\mm\)-rectifiable, Ahlfors regular PI space, and the set of points \(\bar x\in\R^2\) for which \(\Tan_{\bar x}(\R^2,\sfd_F,\mm)\) contains an infinitesimally non-Hilbertian element is exactly \(F\). Indeed, the only modifications needed in the construction are to take \(\mathcal W\) to be the Whitney decomposition of \(\R^2 \setminus F\) and to define the function \(\rho \colon \R^2 \to [1,2]\) as \(2\) on \(F\) and elsewhere via the same definitions as in the proof above. Then the infinitesimal Hilbertianity of the space \((\R^2,\sfd_F,\mm)\) follows from the fact that \(F\) has zero measure, while the infinitesimal non-Hilbertianity of the tangents at \(\bar x \in F\) follows as above. Notice that since \(F\) has zero measure and \(\rho= 2\) on \(F\), the tangent spaces at every point \(\bar x \in F\) are isomorphic to the one obtained in Theorem \ref{thm:example_PI}. 
In the case \(F = \{0\}\) the function \(\rho\) was defined to be \(1\) on \(F\) in order to make \(\rho\) lower semicontinuous. This allowed the soft argument via uniform convergence leading to the existence of \(n(k)\). On a general \(F\) we cannot define \(\rho\) to be identically \(1\), as we might then fail to be infinitesimally non-Hilbertian at the tangents. To overcome this, one could, for example, make a more quantitative argument in the lower bound in Lemma \ref{lem:bound_d_rho}.
We chose to formulate Theorem \ref{thm:example_PI} only in the simplest case \(F = \{0\}\) since the more general case contains essentially no new ideas and only slightly complicates the presentation.
\fr}\end{remark}
\appendix
\section{Infinitesimal Hilbertianity and asymptotic cones}\label{app:asympt_cones}
As one might expect, the infinitesimal Hilbertianity condition has little to do with the large scale geometry of the space under consideration.
Indeed, as shown by Theorem \ref{thm:example_asym_cone} below, it is rather easy to construct a `nice' infinitesimally Hilbertian metric measure
space whose asymptotic cone is not infinitesimally Hilbertian. This is a folklore result, which we discuss in details for the reader's usefulness;
similar constructions are typical in homogenisation theory. Theorem \ref{thm:example_asym_cone} could be obtained by constructing a length distance
on \(\R^2\) induced by similar weights as the ones used in Section \ref{s:proof_thm2}. We opted to provide here an alternative and simpler construction.
Before passing to the actual statement, let us briefly remind the relevant terminology.
\smallskip

Let \((\X,\sfd,\mm)\) be a metric measure space. Then we say that a given pointed metric measure space \((\Y,\sfd_\Y,\mm_\Y,q)\)
is a \textbf{pmGH-asymptotic cone} of \((\X,\sfd,\mm)\) provided there exists a sequence of radii \(R_k\nearrow +\infty\) such that
for some (and thus any) point \(p\in{\rm spt}(\mm)\) it holds that
\[
(\X,\sfd/R_k,\mm^{R_k}_p,p)\to(\Y,\sfd_\Y,\mm_\Y,q),\quad\text{ in the pointed measured Gromov--Hausdorff sense.}
\]
Namely, for every \(\varepsilon\in(0,1)\) and \(\mathcal L^1\)-a.e.\ \(R>1\), there exist \(\bar k\in\N\) and a sequence \((\psi^k)_{k\geq\bar k}\)
of Borel mappings \(\psi^k\colon B_{R R_k}(p)\to\Y\) such that the following properties are verified:
\begin{itemize}
\item[\(\rm i'')\)] \(\psi^k(p)=q\),
\item[\(\rm ii'')\)] \(\big|\sfd(x,y)-R_k\,\sfd_\Y\big(\psi^k(x),\psi^k(y)\big)\big|\leq\varepsilon R_k\) holds for every \(x,y\in B_{R R_k}(p)\),
\item[\(\rm iii'')\)] \(B_{R-\varepsilon}(q)\) is contained in the open \(\varepsilon\)-neighbourhood of \(\psi^k\big(B_{R R_k}(p)\big)\),
\item[\(\rm iv'')\)] \(\mm\big(B_{R_k}(p)\big)^{-1}\psi^k_\#\big(\mm|_{B_{R R_k}(p)}\big)\rightharpoonup\mm_\Y|_{B_R(q)}\) as \(k\to\infty\)
in duality with the space of bounded continuous functions \(f\colon\Y\to\R\) having bounded support.
\end{itemize}
\begin{theorem}\label{thm:example_asym_cone}
There exists an infinitesimally Hilbertian, \(\mm\)-rectifiable, Ahlfors regular PI space having a unique, infinitesimally non-Hilbertian asymptotic cone.
\end{theorem}
\begin{proof}
We endow the grid \(\X\coloneqq(\mathbb Z\times\R)\cup(\R\times\mathbb Z)\) in the plane \(\R^2\) with the distance \(\sfd\),
given by \(\sfd(a,b)\coloneqq\|a-b\|_\infty\) for every \(a,b\in\X\), and with the measure \(\mm\coloneqq\mathcal H^1_\sfd|_\X\).
Consider also the space \(\big(\X\times[0,1],\sfd\times\sfd_{\rm Eucl},\mm\otimes\mathcal L^1|_{[0,1]}\big)\), where \(\mm\otimes\mathcal L^1|_{[0,1]}\)
stands for the product measure and
\[
(\sfd\times\sfd_{\rm Eucl})\big((a,t),(b,s)\big)\coloneqq\sqrt{\sfd(a,b)^2+|t-s|^2},\quad\text{ for every }a,b\in\X\text{ and }t,s\in[0,1].
\]
It is easy to see that \(\big(\X\times[0,1],\sfd\times\sfd_{\rm Eucl},\mm\otimes\mathcal L^1|_{[0,1]}\big)\)
is \((\mm\otimes\mathcal L^1|_{[0,1]})\)-rectifiable, \(2\)-Ahlfors regular, and PI. Using Proposition \ref{prop:equiv_iH_rect}, one can deduce that
\(\big(\X\times[0,1],\sfd\times\sfd_{\rm Eucl},\mm\otimes\mathcal L^1|_{[0,1]}\big)\) is infinitesimally Hilbertian. Observe also that, by virtue of the
fact that \([0,1]\) is bounded, the spaces \(\big(\X\times[0,1],\sfd\times\sfd_{\rm Eucl},\mm\otimes\mathcal L^1|_{[0,1]}\big)\) and \((\X,\sfd,\mm)\)
have the same pmGH-asymptotic cones. Therefore,
to conclude it suffices to prove the following claim: the unique asymptotic cone of \((\X,\sfd,\mm)\)
is given by the infinitesimally non-Hilbertian space \(\big(\R^2,\|\cdot\|_\infty,8^{-1}\mathcal L^2,0\big)\). To this aim, fix any \(\varepsilon\in(0,1)\), \(R>1\), and
\(R_k\nearrow +\infty\). We define the Borel maps \(\psi^k\colon B_{RR_k}^\sfd(0)\to\R^2\) as \(\psi^k(a)\coloneqq a/R_k\) for every \(a\in B_{RR_k}^\sfd(0)\).
Our goal is to show that the sequence \((\psi^k)_{k\in\N}\) verifies the items \(\rm i'')\), \(\rm ii'')\), \(\rm iii'')\), and \(\rm iv'')\) above,
with target \(\big(\R^2,\|\cdot\|_\infty,8^{-1}\mathcal L^2,0\big)\).\\
{\color{blue}\(\rm i'')\)} \(\psi^k(0)=0\) by construction.\\
{\color{blue}\(\rm ii'')\)} It follows from the fact that \(\psi^k\) is an isometry from \(\big(B_{RR_k}^\sfd(0),\sfd\big)\) to \(\big(\R^2,R_k\|\cdot\|_\infty\big)\).\\
{\color{blue}\(\rm iii'')\)} Pick \(\bar k\in\N\) so that \(1/R_{\bar k}<\varepsilon\). Let \(v\in B_{R-\varepsilon}^{\|\cdot\|_\infty}(0)\) and \(k\geq\bar k\) be given.
Since \(R_k v\in B_{RR_k}^{\|\cdot\|_\infty}(0)\), we can find \(a\in \X\cap B_{RR_k}^{\|\cdot\|_\infty}(0)=B_{RR_k}^\sfd(0)\) with \(\|a-R_k v\|_\infty<1\).
This yields \(\big\|\psi^k(a)-v\big\|_\infty<\varepsilon\), thus accordingly \(B_{R-\varepsilon}^{\|\cdot\|_\infty}(0)\) is contained in the
\(\varepsilon\)-neighbourhood of \(\psi^k\big(B_{RR_k}^\sfd(0)\big)\), as desired.\\
{\color{blue}\(\rm iv'')\)} For any \(i,j\in\mathbb Z\) and \(k\in\N\), we define the sets \(Q_{ij}\coloneqq(i-2^{-1},i+2^{-1})\times(j-2^{-1},j+2^{-1})\) and
\(S_k\coloneqq\bigcup_{|i|,|j|<\lfloor RR_k\rfloor}Q_{ij}\), where \(\lfloor\lambda\rfloor\in\N\) stands for the integer part of \(\lambda\in[0,+\infty)\).
Notice that \(\mm\big(B_{\lfloor RR_k\rfloor+1}^{\|\cdot\|_\infty}(0)\setminus S_k\big)=20\lfloor RR_k\rfloor-18\eqqcolon a_k\) and thus
\(\mm\big(B_{RR_k}^\sfd(0)\setminus S_k\big)\leq a_k\). Moreover, calling \(\tilde S_k\coloneqq S_k/R_k\), we have
\(\mathcal L^2\big(B_R^{\|\cdot\|_\infty}(0)\setminus\tilde S_k\big)\leq 8R\big((2R_k)^{-1}+R-\lfloor RR_k\rfloor/R_k\big)\eqqcolon b_k\).
Fix any bounded, continuous function \(f\colon\R^2\to[0,+\infty)\) having compact support. Pick \(C>0\) such that \(f\leq C\).
Given that \(\X\cap B_{\lfloor R_k\rfloor}^{\|\cdot\|_\infty}(0)\subset B_{R_k}^\sfd(0)\) and \(\mm\big(B_{\lfloor R_k\rfloor}^{\|\cdot\|_\infty}(0)\big)
=8\lfloor R_k\rfloor^2-4\lfloor R_k\rfloor\eqqcolon c_k\), we deduce that \(\big|\int_{B_R^{\|\cdot\|_\infty}(0)\setminus\tilde S_k}f\,\d\mathcal L^2\big|\leq C b_k\)
and \(\big|m_k^{-1}\int_{B_{RR_k}^\sfd(0)\setminus S_k}f\circ\psi^k\,\d\mm\big|\leq Ca_k/c_k\), where we set \(m_k\coloneqq\mm\big(B_{R_k}^\sfd(0)\big)\).
Now fix any \(\delta>0\) and choose \(\bar k\in\N\) such that \(\big|f(a)-f(b)\big|\leq\delta\) for every \(a,b\in\R^2\) with \(\|a-b\|_\infty<1/R_{\bar k}\).
Setting \(\rho_k\coloneqq m_k^{-1}\sum_{|i|,|j|<\lfloor RR_k\rfloor}f\big((i,j)/R_k\big)\) for every \(k\geq\bar k\), we obtain that
\[\begin{split}
&\bigg|\frac{1}{m_k}\int f\,\d\psi^k_\#\big(\mm|_{B_{RR_k}^\sfd(0)}\big)-\frac{1}{8}\int_{B_R^{\|\cdot\|_\infty}(0)} f\,\d\mathcal L^2\bigg|\\
\leq\,&\bigg|\frac{1}{m_k}\int_{S_k}f\circ\psi^k\,\d\mm-\frac{1}{8}\int_{\tilde S_k}f\,\d\mathcal L^2\bigg|+C\bigg(\frac{a_k}{c_k}+\frac{b_k}{8}\bigg)\\
\leq\,&\bigg|\frac{1}{m_k}\int_{S_k}f\circ\psi^k\,\d\mm-\rho_k\bigg|+\bigg|\rho_k-\frac{1}{8}\int_{\tilde S_k}f\,\d\mathcal L^2\bigg|
+C\bigg(\frac{a_k}{c_k}+\frac{b_k}{8}\bigg).
\end{split}\]
The first addendum in the last line of the above formula can be estimated as
\[\begin{split}
\bigg|\frac{1}{m_k}\int_{S_k}f\circ\psi^k\,\d\mm-\rho_k\bigg|&\leq
\frac{1}{c_k}\sum_{|i|,|j|<\lfloor RR_k\rfloor}\bigg|\int_{Q_{ij}}f\circ\psi^k\,\d\mm-f\big((i,j)/R_k\big)\bigg|\\
&=\frac{1}{c_k}\sum_{|i|,|j|<\lfloor RR_k\rfloor}\bigg|\fint_{Q_{ij}/R_k}f\,\d\big(\mathcal H^1_{|\cdot|}|_{\X/R_k}\big)-f\big((i,j)/R_k\big)\bigg|\\
&\leq\frac{1}{c_k}\sum_{|i|,|j|<\lfloor RR_k\rfloor}\delta=\frac{\big(2\lfloor RR_k\rfloor-1\big)^2\delta}{c_k},
\end{split}\]
while the second one can be estimated as
\[\begin{split}
\bigg|\rho_k-\frac{1}{8}\int_{\tilde S_k}f\,\d\mathcal L^2\bigg|
&\leq\frac{1}{8R_k^2}\sum_{|i|,|j|<\lfloor RR_k\rfloor}\bigg|\frac{8R_k^2}{m_k}f\big((i,j)/R_k\big)-\fint_{Q_{ij}/R_k}f\,\d\mathcal L^2\bigg|\\
&\leq\frac{1}{8R_k^2}\sum_{|i|,|j|<\lfloor RR_k\rfloor}\bigg(\bigg|\frac{8R_k^2}{m_k}-1\bigg|C+\bigg|f\big((i,j)/R_k\big)-\fint_{Q_{ij}/R_k}f\,\d\mathcal L^2\bigg|\bigg)\\
&\leq\frac{\big(2\lfloor RR_k\rfloor-1\big)^2}{c_k}\bigg(\bigg|\frac{8R_k^2}{m_k}-1\bigg|C+\delta\bigg).
\end{split}\]
Since \(B_{R_k}^\sfd(0)\subset B_{\lfloor R_k\rfloor+1}^{\|\cdot\|_\infty}(0)\), we also have \(m_k\leq\mm\big(B_{\lfloor R_k\rfloor+1}^{\|\cdot\|_\infty}(0)\big)
=8\lfloor R_k\rfloor^2+12\lfloor R_k\rfloor+4\). Then
\[
\lim_{k\to\infty}\frac{a_k}{c_k}=\lim_{k\to\infty}b_k=0,\qquad\limsup_{k\to\infty}\frac{\big(2\lfloor RR_k\rfloor-1\big)^2}{c_k}\leq\frac{R^2}{2},
\qquad\lim_{k\to\infty}\bigg|\frac{8R_k^2}{m_k}-1\bigg|=0.
\]
Therefore, by letting \(k\to\infty\) in the previous estimates we deduce that
\[
\limsup_{k\to\infty}\bigg|\frac{1}{\mm\big(B_{R_k}^\sfd(0)\big)}\int f\,\d\psi^k_\#\big(\mm|_{B_{RR_k}^\sfd(0)}\big)-
\frac{1}{8}\int_{B_R^{\|\cdot\|_\infty}(0)} f\,\d\mathcal L^2\bigg|\leq R^2\delta.
\]
By arbitrariness of \(\delta\) and \(f\), we conclude that \(\mm\big(B_{R_k}^\sfd(0)\big)^{-1}\psi^k_\#\big(\mm|_{B_{RR_k}^\sfd(0)}\big)
\rightharpoonup 8^{-1}\mathcal L^2|_{B_R^{\|\cdot\|_\infty}(0)}\) in duality with bounded continuous functions having compact support, as desired.
\end{proof}
\begin{remark}{\rm
It is easy to show that the space \(\X\times[0,1]\) in Theorem \ref{thm:example_asym_cone} can be additionally required to be a Riemannian manifold.
Indeed, it simply suffices to smoothen the boundary of the set
\[
\bigcup_{x\in\X\times[0,1]}B_{r_x}^{\R^3}(x)\subset\R^3,\quad\text{ where }r_x\coloneqq\frac{1}{\max\{4,|x|\}},
\]
in order to obtain an embedded submanifold with the desired properties.
\fr}\end{remark}
\def\cprime{$'$} \def\cprime{$'$}


\begin{thebibliography}{10}

\bibitem{AmbrosioGigliSavare11-3}
{\sc L.~Ambrosio, N.~Gigli, and G.~Savar{\'e}}, {\em Density of {L}ipschitz
  functions and equivalence of weak gradients in metric measure spaces}, Rev.
  Mat. Iberoam., 29 (2013), pp.~969--996.

\bibitem{AmbrosioGigliSavare11}
\leavevmode\vrule height 2pt depth -1.6pt width 23pt, {\em Calculus and heat
  flow in metric measure spaces and applications to spaces with {R}icci bounds
  from below}, Invent. Math., 195 (2014), pp.~289--391.
  
\bibitem{AmbrosioGigliSavare11-2}
\leavevmode\vrule height 2pt depth -1.6pt width 23pt, {\em Metric measure
  spaces with {R}iemannian {R}icci curvature bounded from below}, Duke Math.
  J., 163 (2014), pp.~1405--1490.

\bibitem{Aubry07}
{\sc E.~Aubry}, {\em Finiteness of {$\pi_1$} and geometric inequalities in
  almost positive {R}icci curvature}, Ann. Sci. \'{E}cole Norm. Sup. (4), 40
  (2007), pp.~675--695.

\bibitem{Bate15}
{\sc D.~Bate}, {\em Structure of measures in {L}ipschitz differentiability
  spaces}, J. Amer. Math. Soc., 28 (2015), pp.~421--482.

\bibitem{Bjorn-Bjorn11}
{\sc A.~Bj{\"o}rn and J.~Bj{\"o}rn}, {\em Nonlinear potential theory on metric
  spaces}, vol.~17 of EMS Tracts in Mathematics, European Mathematical Society
  (EMS), Z\"urich, 2011.

\bibitem{Cheeger00}
{\sc J.~Cheeger}, {\em Differentiability of {L}ipschitz functions on metric
  measure spaces}, Geom. Funct. Anal., 9 (1999), pp.~428--517.

\bibitem{CheegerKleinerSchioppa16}
{\sc J.~Cheeger, B.~Kleiner, and A.~Schioppa}, {\em Infinitesimal structure of
  differentiability spaces, and metric differentiation}, Analysis and Geometry
  in Metric Spaces, 4 (2016), pp.~104--159.

\bibitem{Edwards75}
{\sc D.~A. Edwards}, {\em The {S}tructure of {S}uperspace}, Studies in
  Topology, Academic Press,  (1975).

\bibitem{Eriksson-Bique2019}
{\sc S.~Eriksson-Bique}, {\em Characterizing spaces satisfying {P}oincar\'{e}
  inequalities and applications to differentiability}, Geom. Funct. Anal., 29
  (2019), pp.~119--189.

\bibitem{ErikssonBiqueSoultanis21}
{\sc S.~Eriksson-Bique and E.~Soultanis}, {\em Curvewise characterizations of
  minimal upper gradients and the construction of a Sobolev differential}.
\newblock Preprint, arXiv:2102.08097, 2021.

\bibitem{Fukaya87}
{\sc K.~Fukaya}, {\em Collapsing of {R}iemannian manifolds and eigenvalues of
  {L}aplace operator}, Inventiones mathematicae, 87 (1987), pp.~517--547.

\bibitem{Gigli12}
{\sc N.~Gigli}, {\em On the differential structure of metric measure spaces and
  applications}, Mem. Amer. Math. Soc., 236 (2015), pp.~vi+91.

\bibitem{Gigli14}
\leavevmode\vrule height 2pt depth -1.6pt width 23pt, {\em Nonsmooth
  differential geometry - an approach tailored for spaces with {R}icci
  curvature bounded from below}, Mem. Amer. Math. Soc., 251 (2018), pp.~161.

\bibitem{GigliMondinoSavare13}
{\sc N.~Gigli, A.~Mondino, and G.~Savar\'{e}}, {\em Convergence of pointed
  non-compact metric measure spaces and stability of {R}icci curvature bounds
  and heat flows}, Proc. Lond. Math. Soc. (3), 111 (2015), pp.~1071--1129.

\bibitem{GP16}
{\sc N.~Gigli and E.~Pasqualetto}, {\em Equivalence of two different notions of
  tangent bundle on rectifiable metric measure spaces}.
\newblock To appear in Communications in Analysis and Geometry,
  arXiv:1611.09645, 2016.

\bibitem{GigliTyulenev21}
{\sc N.~Gigli and A.~Tyulenev}, {\em Korevaar--{S}choen's energy on strongly
  rectifiable spaces}, Calculus of Variations and Partial Differential
  Equations, 60 (2021).

\bibitem{Gromov81}
{\sc M.~Gromov}, {\em Groups of polynomial growth and expanding maps},
  Publications Math\'{e}matiques de l'Institut des Hautes \'{E}tudes
  Scientifiques, 53 (1981), pp.~53--78.

\bibitem{HeinonenKoskela98}
{\sc J.~Heinonen and P.~Koskela}, {\em Quasiconformal maps in metric spaces
  with controlled geometry}, Acta Math., 181 (1998), pp.~1--61.

\bibitem{HKST15}
{\sc J.~Heinonen, P.~Koskela, N.~Shanmugalingam, and J.~T. Tyson}, {\em Sobolev
  spaces on metric measure spaces. An approach based on upper gradients},
  vol.~27 of New Mathematical Monographs, Cambridge University Press,
  Cambridge, 2015.

\bibitem{IPS21}
{\sc T.~Ikonen, E.~Pasqualetto, and E.~Soultanis}, {\em Abstract and concrete
  tangent modules on {L}ipschitz differentiability spaces}, Proceedings of the
  American Mathematical Society, 150(1) (2021), pp.~327 --343.

\bibitem{Keith04}
{\sc S.~Keith}, {\em A differentiable structure for metric measure spaces},
  Advances in Mathematics, 183 (2004), pp.~271 --315.

\bibitem{Kei:04}
\leavevmode\vrule height 2pt depth -1.6pt width 23pt, {\em Measurable differentiable structures and the
  {P}oincar\'{e} inequality}, Indiana Univ. Math. J., 53 (2004),
  pp.~1127--1150.

\bibitem{Ketterer21}
{\sc C.~Ketterer}, {\em Stability of metric measure spaces with integral
  {R}icci curvature bounds}, J. Funct. Anal., 281 (2021), pp.~ 109142.

\bibitem{Kirchheim94}
{\sc B.~Kirchheim}, {\em Rectifiable {M}etric {S}paces: {L}ocal {S}tructure and
  {R}egularity of the {H}ausdorff {M}easure}, Proceedings of the American
  Mathematical Society, 121 (1994), pp.~113--123.

\bibitem{LP20}
{\sc D.~Lu\v{c}i\'{c} and E.~Pasqualetto}, {\em Infinitesimal {H}ilbertianity
  of {W}eighted {R}iemannian {M}anifolds}, Canadian Mathematical Bulletin, 63
  (2020), pp.~118--140.

\bibitem{LucicPasqualetto21}
\leavevmode\vrule height 2pt depth -1.6pt width 23pt, {\em Gamma-convergence of {C}heeger
  energies with respect to increasing distances}.
\newblock Preprint, arXiv:2102.06276, 2021.

\bibitem{PetersenWei97}
{\sc P.~Petersen and G.~Wei}, {\em Relative volume comparison with integral
  curvature bounds}, Geom. Funct. Anal., 7 (1997), pp.~1031--1045.

\bibitem{Schioppa16-2}
{\sc A.~Schioppa}, {\em An example of a differentiability space which is
  {P}{I}-unrectifiable}.
\newblock Preprint, arXiv:1611.01615, 2016.

\bibitem{Schioppa16}
\leavevmode\vrule height 2pt depth -1.6pt width 23pt, {\em The {L}ip-lip
  equality is stable under blow-up}, Calculus of Variations and Partial
  Differential Equations, 55 (2016), p.~30.

\bibitem{Shanmugalingam00}
{\sc N.~Shanmugalingam}, {\em Newtonian spaces: an extension of {S}obolev
  spaces to metric measure spaces}, Rev. Mat. Iberoamericana, 16 (2000),
  pp.~243--279.

\end{thebibliography}
\end{document}